\documentclass[12pt,a4paper,twoside]{article}
\usepackage[left=2.54cm,top=2.54cm,right=2.54cm,bottom=2.54cm,bindingoffset=0cm]{geometry}
\usepackage[unicode=true,pdfusetitle,bookmarks=true,bookmarksnumbered=true,bookmarksopen=true,bookmarksopenlevel=1,breaklinks=true,pdfborder={0 0 0},backref=false,colorlinks=true,pdfauthor={Benjamin Thompson},pdftitle={The Khovanov Homology of Rational Tangles}]{hyperref}

\usepackage{amsmath,amsfonts,amsthm,amssymb,mathtools,textcomp}

\numberwithin{equation}{section}
\usepackage{chngcntr}
\counterwithin{figure}{section}

\usepackage{graphicx}

\usepackage{fancyhdr}
\usepackage[dvipsnames]{xcolor}
\hypersetup{
    colorlinks,
    linkcolor={blue!70!black},
    citecolor={violet},
    urlcolor={green!50! black}
}

\usepackage{tikz}
\usetikzlibrary{arrows}
\usetikzlibrary{decorations.markings}

\newcommand{\tangzero}{\begin{scope}[very thick] \draw[-] (-0.5,-0.5) .. controls (-0.2,-0.2) and (0.2,-0.2) .. (0.5,-0.5);\draw[-] (-0.5,0.5) .. controls (-0.2,0.2) and (0.2,0.2) .. (0.5,0.5);\end{scope}}
\newcommand{\tangzerodota}{\tangzero \fill (0,0.28) circle (6pt);}
\newcommand{\tangzerodotb}{\tangzero \fill (0,-0.28) circle (6pt);}

\newcommand{\tanginf}{\begin{scope}[very thick] \draw[-] (-0.5,0.5) .. controls (-0.2,0.2) and (-0.2,-0.2) .. (-0.5,-0.5);\draw[-] (0.5,0.5) .. controls (0.2,0.2) and (0.2,-0.2) .. (0.5,-0.5);\end{scope}}
\newcommand{\tanginfdota}{\tanginf \fill (-0.28,0) circle (6pt);}
\newcommand{\tanginfdotb}{\tanginf \fill (0.28,0) circle (6pt);}
\newcommand{\tanginfdotab}{\tanginfdota \fill (0.28,0) circle (6pt);}

\newcommand{\nodezero}{\begin{tikzpicture}[scale=0.7]\tangzero\end{tikzpicture}}
\newcommand{\nodezerodota}{\begin{tikzpicture}[scale=0.35]\tangzerodota\end{tikzpicture}}
\newcommand{\nodezerodotb}{\begin{tikzpicture}[scale=0.35]\tangzerodotb\end{tikzpicture}}
\newcommand{\nodeinf}{\begin{tikzpicture}[scale=0.7]\tanginf\end{tikzpicture}}
\newcommand{\nodeinfdota}{\begin{tikzpicture}[scale=0.35]\tanginfdota\end{tikzpicture}}
\newcommand{\nodeinfdotb}{\begin{tikzpicture}[scale=0.35]\tanginfdotb\end{tikzpicture}}
\newcommand{\nodeinfdotab}{\begin{tikzpicture}[scale=0.35]\tanginfdotab\end{tikzpicture}}

\newcommand{\saddleh}{\tanginf \draw[very thick,-] (-0.28,0) -- (0.28,0);}
\newcommand{\saddlev}{\tangzero \draw[very thick,-] (0,0.28) -- (0,-0.28);}
\newcommand{\nodesaddleh}{\begin{tikzpicture}[scale=0.35]\saddleh\end{tikzpicture}}

\newcommand{\clownvert}{\begin{scope}[very thick] \draw[-] (-0.5,-1) .. controls (-0.2,-0.7) and (0.2,-0.7) .. (0.5,-1);\draw[-] (-0.5,1) .. controls (-0.2,.7) and (0.2,.7) .. (0.5,1);\draw[very thick] (0,0) circle (0.4cm);\end{scope}}
\newcommand{\clownhorz}{\begin{scope}[very thick] \draw[-] (-1,0.5) .. controls (-0.7,0.2) and (-0.7,-0.2) .. (-1,-0.5);\draw[-] (1,0.5) .. controls (0.7,0.2) and (0.7,-0.2) .. (1,-0.5);\draw[very thick] (0,0) circle (0.4cm);\end{scope}}

\newcommand{\nodeclownv}{\begin{tikzpicture}[scale=0.7]\clownvert\end{tikzpicture}}
\newcommand{\nodeclownh}{\begin{tikzpicture}[scale=0.7]\clownhorz\end{tikzpicture}}

\newcommand{\clownhorzdeloop}{\begin{scope}[very thick] \draw[-] (-1,0.5) .. controls (-0.7,0.2) and (-0.7,-0.2) .. (-1,-0.5);\draw[-] (1,0.5) .. controls (0.7,0.2) and (0.7,-0.2) .. (1,-0.5);\end{scope}}

\newcommand{\nodeclownhdeloop}{\begin{tikzpicture}[scale=0.7]\clownhorzdeloop\end{tikzpicture}}

\newcommand{\clownhspike}{\clownhorz \foreach \w in {-135,-45,45,135} {\draw[very thick,-] (\w:0.4) -- (\w:0.65);}}
\newcommand{\clownhspa}{\clownhspike \fill (-0.78,0) circle (6pt);}
\newcommand{\clownhspb}{\clownhspike \fill (0,0) circle (5pt);}
\newcommand{\clownhspab}{\clownhspa \fill (0,0) circle (5pt);}
\newcommand{\clownhspac}{\clownhspa \fill (0.78,0) circle (6pt);}
\newcommand{\clownhspbc}{\clownhspb \fill (0.78,0) circle (6pt);}

\newcommand{\nodeclownhspab}{\begin{tikzpicture}[scale=0.35]\clownhspab\end{tikzpicture}}
\newcommand{\nodeclownhspac}{\begin{tikzpicture}[scale=0.35]\clownhspac\end{tikzpicture}}
\newcommand{\nodeclownhspbc}{\begin{tikzpicture}[scale=0.35]\clownhspbc\end{tikzpicture}}

\newcommand{\mixhorzzeroinf}{\begin{scope}[very thick]
\draw[-] (1.0,0.5) .. controls (0.7,0.2) and (0.7,-0.2) .. (1,-0.5);
\draw[-] (-1,0.5) .. controls (-0.4,0) and (-0.2,0.5) .. (0,0.5) .. controls (0.4,0.5) and (0.4,-0.5) .. (0,-0.5) .. controls (-0.2,-0.5) and (-0.4,0) .. (-1,-0.5);\end{scope}}

\newcommand{\mixhorzinfzero}{\begin{scope}[very thick]
\draw[-] (-1.0,0.5) .. controls (-0.7,0.2) and (-0.7,-0.2) .. (-1,-0.5);
\draw[-] (1,0.5) .. controls (0.4,0) and (0.2,0.5) .. (0,0.5) .. controls (-0.4,0.5) and (-0.4,-0.5) .. (0,-0.5) .. controls (0.2,-0.5) and (0.4,0) .. (1,-0.5);\end{scope}}

\newcommand{\mixhorzzerozero}{\begin{scope}[very thick]
\draw[-] (-1,0.5) .. controls (-0.4,0) and (-0.2,0.5) .. (0,0.5) .. controls (0.2,0.5) and (0.4,0) .. (1,0.5);
\draw[-] (-1,-0.5) .. controls (-0.4,0) and (-0.2,-0.5) .. (0,-0.5) .. controls (0.2,-0.5) and (0.4,0) .. (1,-0.5);\end{scope}}

\newcommand{\nodemhzeroinf}{\begin{tikzpicture}[scale=0.7]\mixhorzzeroinf\end{tikzpicture}}
\newcommand{\nodemhinfzero}{\begin{tikzpicture}[scale=0.7]\mixhorzinfzero\end{tikzpicture}}
\newcommand{\nodemhzerozero}{\begin{tikzpicture}[scale=0.7]\mixhorzzerozero\end{tikzpicture}}

\newcommand{\zeroinfsad}{\mixhorzzeroinf \draw[very thick,-] (0.3,0) -- (0.78,0);}
\newcommand{\infzerosad}{\mixhorzinfzero \draw[very thick,-] (-0.78,0) -- (-0.3,0);}

\newcommand{\nodezeroinfsad}{\begin{tikzpicture}[scale=0.35]\zeroinfsad\end{tikzpicture}}
\newcommand{\nodeinfzerosad}{\begin{tikzpicture}[scale=0.35]\infzerosad\end{tikzpicture}}

\newcommand{\clownhsadl}{\clownhorz \draw[very thick,-] (-0.4,0) -- (-0.78,0);}
\newcommand{\clownhsadr}{\clownhorz \draw[very thick,-] (0.4,0) -- (0.78,0);}

\newcommand{\nodeclownsadl}{\begin{tikzpicture}[scale=0.35]\clownhsadl\end{tikzpicture}}
\newcommand{\nodeclownsadr}{\begin{tikzpicture}[scale=0.35]\clownhsadr\end{tikzpicture}}

\newcommand{\zdashz}{\begin{tikzpicture}[very thick,scale=.7]\draw (0,0) -- (1,0);\filldraw[fill=white,very thick] (0,0) circle (4.5pt); \filldraw[fill=white,very thick] (1,0) circle (4.5pt);\end{tikzpicture}}

\newcommand{\circleprop}{\begin{tikzpicture}[semithick,scale=0.13]\draw (0,0) circle (1cm);\end{tikzpicture}}

\newcommand{\nodedottedcircle}{\begin{tikzpicture}\draw[very thick] (0,0) circle (0.175); \fill[scale=0.35] (-0.45,0) circle (6pt);\end{tikzpicture}}

\newcommand{\miniinf}{\begin{tikzpicture}[scale=0.3, line width=0.7pt,rotate around={90:(0.5,0.5)}]\draw (0,0) .. controls (0.5,0.4) and (0.5,0.4) .. (1,0);\draw (0,1) .. controls (0.5,0.6) and (0.5,0.6) .. (1,1);\end{tikzpicture}}
\newcommand{\minizero}{\begin{tikzpicture}[scale=0.3, line width=0.7pt]\draw (0,0) .. controls (0.5,0.4) and (0.5,0.4) .. (1,0);\draw (0,1) .. controls (0.5,0.6) and (0.5,0.6) .. (1,1);\end{tikzpicture}}

\newcommand{\Z}{\mathbb{Z}}

\newcommand{\Q}{\mathbb{Q}}

\newcommand{\Oo}{\mathcal{O}}
\newcommand{\inv}{^{-1}}
\newcommand{\la}{\langle}
\newcommand{\ra}{\rangle}

\DeclareMathOperator{\Mat}{Mat}

\DeclareMathOperator{\Kom}{Kom}

\DeclareMathOperator{\BN}{BN}

\DeclareMathOperator{\qdim}{qdim}

\newcommand{\Cc}{\mathcal{C}}

\newcommand{\Cob}{\Cc ob^3}

\newcommand{\Cobbl}{\Cob_{\bullet/l}}

\newcommand{\Kb}[1]{\text{\textlbrackdbl} #1 \text{\textrbrackdbl}}
\newcommand{\Kbp}[1]{\text{\emph{\textlbrackdbl}} #1 \text{\emph{\textrbrackdbl}}}

\newcommand{\Kh}{Kh}

\newtheoremstyle{dotless_one}{}{}{\itshape}{}{\bfseries}{}{ }{}
\theoremstyle{dotless_one}

\newtheorem{thm}{Theorem}[section]
\newtheorem{coro}[thm]{Corollary}
\newtheorem{lemma}[thm]{Lemma}
\newtheorem{prop}[thm]{Proposition}

\newtheoremstyle{dotless_two}{}{}{\upshape}{}{\bfseries}{}{ }{}
\theoremstyle{dotless_two}
\newtheorem{defin}[thm]{Definition}
\newtheorem{example}[thm]{Example}

\makeatletter
\renewenvironment{proof}[1][\proofname]{\par
  \pushQED{\qed}%
  \normalfont \topsep6\p@\@plus6\p@\relax
  \trivlist
  \item[\hskip\labelsep
        \itshape
   #1\@addpunct{:}]\ignorespaces% DELETED
}{%
  \popQED\endtrivlist\@endpefalse
}
\makeatother

\usepackage{enumitem}
\setlist{noitemsep}

\title{\bfseries \Large Khovanov complexes of rational tangles}
\author{Benjamin Thompson}
\date{}

\begin{document}
\pagestyle{fancy}
\pagenumbering{arabic}
\setcounter{page}{1}

\maketitle

\begin{abstract}
We show that the Khovanov complex of a rational tangle has a very simple representative whose backbone of non-zero morphisms forms a zig-zag. Furthermore, this minimal complex can be computed quickly by an inductive algorithm. (For example, we calculate $\Kh(8_2)$ by hand.) We find that the bigradings of the subobjects in these minimal complexes can be described by matrix actions, which after a change of basis is the reduced Burau representation of $B_3$.
\end{abstract}

\section{Introduction}
%!TEX root = ..\main.tex

In \cite{TangleAndCobords}, Dror Bar-Natan reformulated the Khovanov homology of tangles in such a way that the Khovanov bracket was essentially local. He then introduced a delooping isomorphism in \cite{FastKh} which together with some slight modifications to the underlying theory, was able to rapidly simplify the Khovanov complex of a link allowing the homology groups to be computed quickly.

In this note we show that the Khovanov complex of a rational tangle has an unexpectedly simple representative (the minimal complex), which is a zig-zag complex (Theorem~\ref{thm-main}). Furthermore, the bigrading information of the subobjects in this minimal complex can be described by matrix actions which after a change of basis is the Burau representation (Corollary~\ref{coro-braid-rels}). The ramifications of this observation remain unclear.

The proof is inductive in that it relies on the fact that every rational tangle can be constructed by adding crossings sequentially, and that each time a crossing is added, we simplify the result to a zig-zag complex. The proof that the bigradings are described by matrix actions is essentially a counting argument based on the previous proof.

We describe our definition of a zig-zag complex below, but the core idea is that the backbone of non-zero maps between the subobjects of the complex collectively form a zig-zag.
\[
	\begin{tikzpicture}[auto,scale=0.6]
\node (2) at (0,0) {$\Z$};
\node (3a) at (3,1) {$\Z$};
\node (3m) at (3,0) {$\oplus$};
\node (3b) at (3,-1) {$\Z$};
\node (4a) at (6,1) {$\Z$};
\node (4m) at (6,0) {$\oplus$};
\node (4b) at (6,-1) {$\Z$};
\node (5) at (9,0) {$\Z$};

\draw[->] (2) to node[swap] {\footnotesize $2$} (3b);
\draw[->] (3b) to node[swap] {\footnotesize $2$} (4b);
\draw[->] (3a) to node {\footnotesize $3$} (4b);
\draw[->] (3a) to node {\footnotesize $4$} (4a);
\draw[->] (4a) to node {\footnotesize $2$} (5);
\end{tikzpicture}
\]

\begin{defin}\label{def-zig-zag}
	Let $\Cc$ be a preadditive category, and let $\Mat(\Cc)$ be the additive closure of $\Cc$ (i.e. the category of formal direct sums of objects in $\Cc$ with composition of morphisms given by matrix multiplication).
	Let $(\Omega,\partial)$ be a representative of a complex in $\Kom(\Mat(\Cc))$, the category of complexes in $\Mat(\Cc)$ considered up homotopy equivalence. Fix a particular direct sum decomposition of $(\Omega, \partial)$.
	We call a subobject $\Omega_j^i$ of the complex a \emph{z-end} if there is precisely one non-zero map from either some subobject of $\Omega^{i-1}$ to it, or from it to some subobject of $\Omega^{i+1}$. We call a subobject $\Omega_j^i$ a \emph{z-middle} of the complex if either: there is precisely one non-zero map from some subobject of $\Omega^{i-1}$ to $\Omega_j^i$ and precisely one non-zero map from $\Omega_j^i$ to some subobject of $\Omega^{i+1}$; or the map from $\Omega^{i-1}$ to $\Omega_j^i$ is zero and there are precisely two non-zero maps from $\Omega_j^i$ to subobjects of $\Omega^{i+1}$; or there are precisely two non-zero maps from subobjects of $\Omega^{i-1}$ to $\Omega_j^i$ and the map from $\Omega_j^i$ to $\Omega^{i+1}$ is zero.

	We say $(\Omega,\partial)$ is a \emph{zig-zag complex} if every subobject in the complex is either a $z$-end or a $z$-middle, and precisely two subobjects are $z$-ends.
\end{defin}

The structure of the article is as follows. In Section~\ref{sec-RTs} we quickly review the essentials of rational tangles, and then in Section~\ref{sec-Kh} review the Khovanov homology theory of tangles we work with. The main content begins in section~\ref{sec-KhRts}, where we compute the minimal complex of an integer tangle, introduce a `square' isomorphism, then combine the two to prove Theorem~\ref{thm-main}, the first of our main results. After some counting, we then establish Theorem~\ref{thm-mat-actions}, our other result regarding matrix actions.

\subsection*{Acknowledgements}
The bulk of this paper derives from an Honours thesis of the same name which would not have been possible without much help and support. I wish to thank my excellent supervisor Scott Morrison for providing the topic and general direction of the project, and Tony Licata, whose idea regarding matrix actions lead to Section~\ref{sec-burau}.

\section{Rational tangles primer}\label{sec-RTs}
%!TEX root = ..\main.tex
We briefly describe the essentials concerning rational tangles. More detailed exposition can be found, for instance, in \cite{KauffRTs}.

Rational tangles are a subclass of tangles with four boundary points (4-point tangles). On these, one can define binary operations referred to as \emph{addition} $(+)$ and \emph{multiplication} $(\ast)$ illustrated below in Figure~\ref{fig-def-add_mult}.

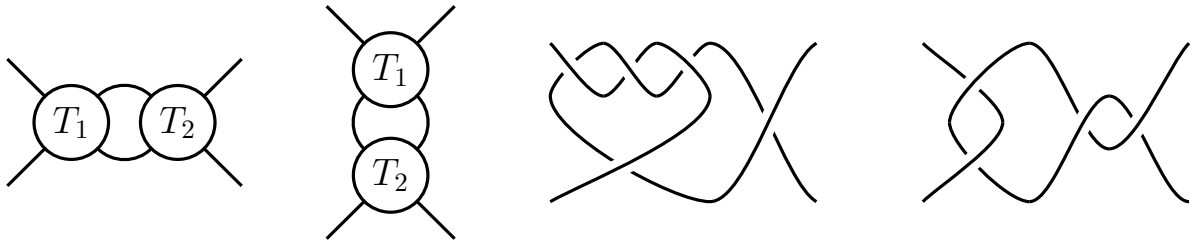
\begin{figure}[ht!]
	\centering
	\begin{tikzpicture}[very thick,scale=0.7]
%% +, * tangle diagrams
\begin{scope}[xshift=-5cm,rotate around={90:(0,0)}]
\draw (0,1) circle (0.7cm);
\draw (0,-1) circle (0.7cm);
\node (T1) at (0,1) {\large $T_1$};
\node (T2) at (0,-1) {\large $T_2$};	
\begin{scope}[yshift=-1cm] 
\draw (0,2) ++
(-45:.7)  to [out=-45,in=45,looseness=1]
(45:0.7);
\draw (0,2) ++ (-135:0.7) to [out=-135,in=135,
looseness=1] (135:0.7);
\end{scope}
\draw (0,1) ++ (135:0.7) -- ++ (135:1);
\draw (0,1) ++ (45:0.7) -- ++ (45:1);
\draw (0,-1) ++ (-135:0.7) -- ++ (-135:1);
\draw (0,-1) ++ (-45:0.7) -- ++ (-45:1);
\end{scope}

\begin{scope}
\draw (0,1) circle (0.7cm);
\draw (0,-1) circle (0.7cm);
\node (T1) at (0,1) {\large $T_1$};
\node (T2) at (0,-1) {\large $T_2$};	
\begin{scope}[yshift=-1cm] 
\draw (0,2) ++
(-45:.7)  to [out=-45,in=45,looseness=1]
(45:0.7);
\draw (0,2) ++ (-135:0.7) to [out=-135,in=135,
looseness=1] (135:0.7);
\end{scope}
\draw (0,1) ++ (135:0.7) -- ++ (135:1);
\draw (0,1) ++ (45:0.7) -- ++ (45:1);
\draw (0,-1) ++ (-135:0.7) -- ++ (-135:1);
\draw (0,-1) ++ (-45:0.7) -- ++ (-45:1);
\end{scope}

%% rational tangle examples
\begin{scope}[yshift=-1.5cm,xshift=3cm]
\begin{scope}[looseness=0.5]
\draw (0,3) [out=-45,in=180] to (1,2) [out=0,in=180] to (2,3) [out=0,in=90] to (3,2) [out=-90,in=30] to (0,0);
\draw (5,0) [out=150,in=0] to (3,3) [out=180,in=0] to (2,2) [out=180,in=0] to (1,3) [out=180,in=90] to (0,2) [out=-90,in=180] to (3,0) [out=0,in=-150] to (5,3);

\draw (0,3) [white,line width=7pt,out=-45,in=180] to (1,2);
\draw (0,3) [out=-45,in=180] to (1,2);
\draw (1,3) [white,line width=7pt,out=0,in=180] to (2,2);
\draw (1,3) [out=0,in=180] to (2,2);
\draw (2,3) [white,line width=7pt,out=0,in=90] to (3,2);
\draw (2,3) [out=0,in=90] to (3,2);
\draw (0,0) [white,line width=7pt,out=30,in=-90] to (3,2);
\draw (0,0) [out=30,in=-90] to (3,2);
\draw (3,0) [white,line width=7pt,out=0,in=-150] to (5,3);
\draw (3,0) [out=0,in=-150] to (5,3);
\end{scope}

\begin{scope}[xshift=7cm,looseness=0.5,rotate around={180:(2.5,1.5)}]
\draw (0,0) [out=45,in=180] to (1.5,2) [out=0,in=180] to (3,0) [out=0,in=-90] to (4.5,1.5) [out=90,in=0] to (3,3) [out=180,in=0] to (1.5,1) [in=0,out=180] to (0,3);
\draw (5,0) [out=135,in=-90] to (3.5,1.5) [out=90,in=-135] to (5,3);

\draw (0,0) [white,line width=7pt,out=45,in=180] to (1.5,2);
\draw (0,0) [out=45,in=180] to (1.5,2);
\draw (1.5,1) [white,line width=7pt,out=0,in=180] to (3,3);
\draw (1.5,1) [out=0,in=180] to (3,3);
\draw (3,0) [white,line width=7pt,out=0,in=-90] to (4.5,1.5);
\draw (3,0) [out=0,in=-90] to (4.5,1.5);
\draw (3.5,1.5) [white,line width=7pt,out=90,in=-135] to (5,3);
\draw (3.5,1.5) [out=90,in=-135] to (5,3);
\end{scope}
\end{scope}

\end{tikzpicture}
	\caption{\textsc{from left to right} The sum ($+$) and product ($*$) of 4-point tangles $T_1$ and $T_2$. The rational tangles $\la -3,1,1 \ra$ and $\la 1,1,2 \ra$.}
	\label{fig-def-add_mult}
\end{figure}

\begin{defin}
\label{def-twist_form}
Let $[0]$ and $[\infty]$ refer to the $\minizero$ and $\miniinf$ tangles respectively. Let $[\pm 1]$ refer to the single-crossing 4-point tangles comprised of diagonal strands where the sign indicates the gradient of the overcrossing. Then a tangle is \emph{rational} if it is isotopic to a tangle created from $[0]$ or $[\infty]$ by a finite sequence of additions and multiplications with the tangles $[\pm 1]$.
\end{defin}

Using the notation $[n]:= [1] + \ldots + [1]$ and $[\overline{n}] = [1] * \ldots * [1]$, we can write a rational tangle as an expression
\[
	\left( \ldots ((([a_1] * [\overline{a_2}]) + [a_3] ) * [\overline{a_4}] ) + \ldots * [\overline{a_{n-1}}]) + [a_n]
	 \right.
\]
where $a_i$ are integers. We will abbreviate such an expression as $\la a_1, \ldots, a_n \ra $.

The reasoning behind the name of these tangles stems from the fact that each can be assigned a rational number which is a complete invariant of the tangle.

\begin{defin}\label{def-tang-frac}
Let $T$ be a rational tangle isotopic to $\la a_n, a_{n-1},\ldots a_1 \ra$. Define the \emph{fraction} or \emph{tangle fraction} $F(T)$ of $T$ to be the rational number
\begin{equation}
	F(T) = a_1   + \frac{1}{a_2      +\frac{1}{a_3       + \ldots + \frac{1}{a_{n-1} + \frac{1}{a_n}}}}.
\end{equation}
If $T = [\infty]$ we define $F([\infty]) = \infty$.
\end{defin}

\begin{thm}
\label{thm-RT_class}
The tangle fraction is well-defined, and two rational tangles are isotopic if and only if they have the same fraction.
\end{thm}

\begin{example}
The rational tangles in Figure~\ref{fig-def-add_mult} above both have a tangle fraction equal to $5/2$, so are in fact isotopic. This may not be immediately evident from inspection.
\end{example}

\begin{defin}
A continued fraction such as that in Defintion~\ref{def-tang-frac} is said to be in \emph{canonical form} if $n$ is odd and either $a_1 \geq 0$ and $a_i>0$, or $a_1\leq 0$ and $a_i<0$ (for $i \ne 1$).
\end{defin}

It is not difficult to see that every rational number has a unique continued fraction in canonical form, which gives us the following result.

\begin{coro}\label{coro-canonical-form}
Every rational tangle other than $[\infty ]$ has a unique canonical form.
\end{coro}

Given a rational tangle, one can close the ends in several ways to obtain a \emph{rational link}. The \emph{numerator closure} $N(\cdot)$ is obtained by horizontally joining together the top and bottom pairs. The Hopf link, for example, can be expressed as $N([2])$. The tangles exhibited so far have been unoriented, but of course this need not be the case. For notational purposes in Section~\ref{sec-KhRts}, we note that oriented rational tangles that have oriented numerator closures are isotopic to one of two forms, illustrated below. We refer to these as \emph{type I} and \emph{type II} (oriented) rational tangles.

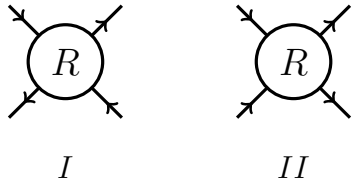
\[
	\begin{tikzpicture}[auto,semithick,decoration={markings,mark=at position 0.5 with {\arrow{>}}}]

\begin{scope}[xshift=-3cm,scale=0.7,very thick]
\draw (0,0) circle (0.7cm);
\node (Tl) at (0,0) {\large $R$};	
\draw[postaction={decorate}] (135:0.7)++ (135:0.8) -- ++ (135:-0.8);
\draw[postaction={decorate}] (45:0.7) -- ++ (45:0.8);
\draw[postaction={decorate}] (-135:0.7) -- ++ (-135:0.8);
\draw[postaction={decorate}] (-45:0.7) ++ (-45:.8) -- ++ (-45:-.8);
\node (l) at (0,-2) {$I$};
\end{scope}

\begin{scope}[xshift=0cm,scale=0.7,very thick]
\draw (0,0) circle (0.7cm);
\node (Tl) at (0,0) {\large $R$};	
\draw[postaction={decorate}] (135:0.7)++ (135:0.8) -- ++ (135:-0.8);
\draw[postaction={decorate}] (45:0.7) -- ++ (45:0.8);
\draw[postaction={decorate}] (-135:0.7) ++ (-135:0.8) -- ++ (-135:-0.8);
\draw[postaction={decorate}] (-45:0.7) -- ++ (-45:.8);
\node (l) at (0,-2) {$II$};
\end{scope}
\end{tikzpicture}
\]

\section{A recap of local Khovanov homology}\label{sec-Kh}
We will assume the reader is familiar with Bar-Natan's definition of the Khovanov complex for tangles \cite{TangleAndCobords}, as well as his local `dotted' theory with its machinery for fast computations \cite{FastKh}. Nonetheless, we quickly recall the essentials below to fix notation.

\begin{defin}
	The category $\Cob_0(\partial T)$ is defined to have as objects formally graded smoothings (simple curves in the plane) with boundary $\partial T$. The Hom-sets $\Cob_0(\partial T)(\Oo \rightarrow \Oo')$ between two smoothings in $\Cob_0(\partial T)$ consist of formal linear combinations of all oriented two-dimensional surfaces embedded into a cylinder with boundary $(\Oo \times \{ 0 \}) \cup (\partial T \times [0,1]) \cup (\Oo' \times \{ 1 \})$, considered up to isotopy. Composition is defined by placing one cobordism on top of the other and vertically renormalizing the result, illustrated below.
	
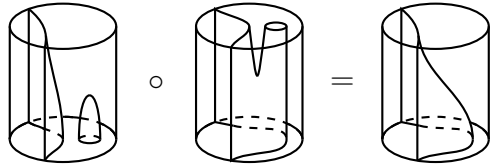
\[
		\begin{tikzpicture}[yscale=0.5,xscale=0.7,thick]

\begin{scope}[xshift=3.5cm]
\draw[dashed] (-1,0) arc (180:0:1cm and 0.6cm);
\draw (-1,0) arc (-180:0:1cm and 0.6cm);
\draw (0,3) ellipse (1cm and 0.6cm);
\draw (-1,0) -- (-1,3);
\draw (1,0) -- (1,3);
\begin{scope}[yshift=3cm]
\draw [out=-50,in=80] (130:1cm and 0.6cm) ++(0,-3) -- ++(0,3) to (-110:1cm and 0.6cm) -- ++(0,-3);
\end{scope}
\draw (0.7,0) [out=-90,in=40,looseness=1] to
(-110:1cm and 0.6cm); -- ++ (0,3) [out=70,in=-90] to  ++(70:1cm and 0.6cm) [out=90,in=-50,looseness=0.6] to ++ (130:1cm and 0.6cm) -- ++(0,-3);
\draw [out=-20,in=90,dashed] (115:0.7cm and 0.4cm) to (0.7,0);
\draw (-1,0) arc(180:130:1cm and 0.6cm) -- (115:0.7cm and 0.4cm);
\draw (1,0) arc (0:45:1cm and 0.6cm);

\draw [out=90,in=-85] (0.7,0) to ++ (-1.02,2.75);
\end{scope}

\begin{scope}[xshift=-3.5cm]
\draw[dashed] (-1,0) arc (180:0:1cm and 0.6cm);
\draw (-1,0) arc (-180:0:1cm and 0.6cm);
\draw (0,3) ellipse (1cm and 0.6cm);
\draw (-1,0) -- (-1,3);
\draw (1,0) -- (1,3);
\draw (-110:1cm and 0.6cm) [out=70,in=-90] to  ++(70:1cm and 0.6cm); \draw [out=90,in=-50,looseness=0.6,dashed] (0,0) to ++ (130:1cm and 0.6cm);
\draw (130:1cm and 0.6cm) -- ++(0,3);
\draw (-1,0) arc(180:130:1cm and 0.6cm) [out=-50] to ++(-33:0.35);
\draw (70:1cm and 0.6cm) arc (70:90:1cm and 0.6cm) ;

\draw (1,0) arc (0:45:1cm and 0.6cm);
\draw[dashed] (0.3,0) arc (180:0:0.2cm and 0.1cm);
\draw (0.3,0) arc (-180:0:0.2cm and 0.1cm);
\draw [out=90,in=90,looseness=10] (0.3,0) to (0.7,0);

\draw [out=90,in=-85] (0,0) to ++ (-0.32,2.75);

\begin{scope}[yshift=3cm]
\draw [out=-50,in=80] (130:1cm and 0.6cm) to (-110:1cm and 0.6cm) -- ++(0,-3);
\end{scope}
\end{scope}

\begin{scope}[xshift=0cm]
\draw[dashed] (-1,0) arc (180:0:1cm and 0.6cm);
\draw (-1,0) arc (-180:0:1cm and 0.6cm);
\draw (0,3) ellipse (1cm and 0.6cm);
\draw (-1,0) -- (-1,3);
\draw (1,0) -- (1,3);
\draw (0.7,3) -- (0.7,0) [out=-90,in=40,looseness=1] to
(-110:1cm and 0.6cm) -- ++ (0,3) [out=70,in=-90] to  ++(70:1cm and 0.6cm) [out=90,in=-50,looseness=0.6] to ++ (130:1cm and 0.6cm) -- ++(0,-3);
\draw [out=-20,in=90,dashed] (115:0.7cm and 0.4cm) to (0.7,0);
\draw (-1,0) arc(180:130:1cm and 0.6cm) -- (115:0.7cm and 0.4cm);
\draw (1,0) arc (0:45:1cm and 0.6cm);
\draw (0.5,3) ellipse (0.2cm and 0.1cm);
\draw (0,3) [out=-85,in=-95,looseness=15] to (0.3,3);
%\draw[white] (0,0) ellipse (0.7cm and 0.4cm);
\end{scope}

\node (c) at (-1.75,1.5) {$\circ$};
\node (e) at (1.75,1.5) {$=$};

\end{tikzpicture}
	\]
	The category $\Mat(\Cob_0(\partial T))$ is simply the category of formal direct sums of objects in $\Cob_0(\partial T))$, with composition of morphisms given by matrix multiplication.
\end{defin}

The bigraded Khovanov complex for tangles is created by following the general recipe: given a tangle $T$ with $n$ crossings, once creates a $n$-dimensional cube with smoothings as vertices and cobordisms as edges. (To save rainforests, we won't replicate the diagrams illustrating this cube --- which are in our view the easiest way to understand the construction -- \cite{TangleAndCobords} contains several.) The cube is then flattened to produce a complex in $\Mat(\Cob_0(\partial T))$. This Khovanov complex $\Kb{T}$ is not yet a tangle invariant, but can be turned into one by possibly extending the Hom-sets of $\Cob_0$, and quotienting these out with one of several alternative sets of relations. (Each of which each give different flavours of Khovanov homology.)

In this note we use the `dotted theory', so extend $\Cob_0$ by allowing the cobordisms to contain a finite number of dots, then quotient out the Hom-sets using the relations below. We refer to this quotiented category as $\Cobbl$.
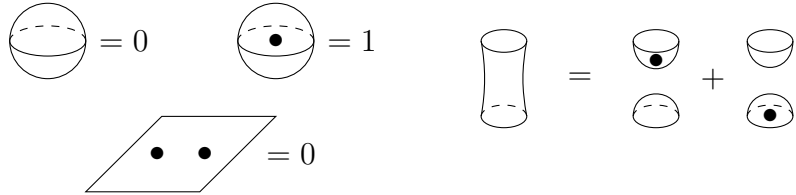
\[
	\begin{tikzpicture}

\begin{scope}[xshift=-6cm]
\draw (0,1) circle (0.5cm);
\node (S') at (1,1) {$= 0$};
\draw[dashed] (0.5,1) arc(0:180:0.5cm and 0.2cm);
\draw (-0.5,1) arc(180:360:0.5cm and 0.2cm);
\end{scope}

\begin{scope}[xshift=-3cm]
\draw (0,1) circle (0.5cm);
\node (S') at (1,1) {$= 1$};
\draw[dashed] (0.5,1) arc(0:180:0.5cm and 0.2cm);
\draw (-0.5,1) arc(180:360:0.5cm and 0.2cm);
\node (sb) at (0,1) {$\bullet$};
\end{scope}

\begin{scope}[shift={(-5,-0.5)}]
\draw (-0.5,-0.5) -- ++ (1,1) -- ++ (1.5,0) -- ++ (-1,-1) -- cycle;
\node (pds) at (0.75,0) {$\bullet \quad \bullet$};
\node (p) at (2.2,0) {$= 0$};
\end{scope}

\begin{scope}[shift={(0,-0.5)}]
\begin{scope}[xshift=0cm]
\draw[dashed] (0.3,0.5) arc(0:180:0.3cm and 0.15cm);
\draw (0.3,0.5) arc(0:-180:0.3cm and 0.15cm);
\draw (0,1.5) ellipse (0.3cm and 0.15cm);
\draw [out=-100,in=100] (0.3,1.5) to (0.3,0.5);
\draw [out=-80,in=80,looseness=0.9] (-0.3,1.5) to (-0.3,0.5);
\end{scope}
\begin{scope}[xshift=2cm]
\draw[dashed] (0.3,0.5) arc(0:180:0.3cm and 0.15cm);
\draw (0.3,0.5) arc(0:-180:0.3cm and 0.15cm);
\draw (0,1.5) ellipse (0.3cm and 0.15cm);
\draw [out=-90,in=-90,looseness=2.1] (-0.3,1.5) to (0.3,1.5);
\draw [out=90,in=90,looseness=1.7] (-0.3,0.5) to (0.3,0.5);
\node (bt) at (0,1.23) {$\bullet$};
\end{scope}

\begin{scope}[xshift=3.5cm]
\draw[dashed] (0.3,0.5) arc(0:180:0.3cm and 0.15cm);
\draw (0.3,0.5) arc(0:-180:0.3cm and 0.15cm);
\draw (0,1.5) ellipse (0.3cm and 0.15cm);
\draw [out=-90,in=-90,looseness=2] (-0.3,1.5) to (0.3,1.5);
\draw [out=90,in=90,looseness=1.7] (-0.3,0.5) to (0.3,0.5);
\node (bt) at (0,0.5) {$\bullet$};
\end{scope}
\node (e) at (1,1) {$=$};
\node (p) at (2.75,1) {$+$};
\end{scope}

\end{tikzpicture}
\]
We call the last relation \emph{neck cutting}.

This `dotted theory' is more refined than the default theory in \cite{TangleAndCobords} (in that the corbordism quotienting relations there also hold in the dotted one), but now we can consider homotopy equivalences involving dotted cobordisms. This allows us to simplify $\Kb{T}$ in ways not previously possible. Our results depend on two such isomorphisms, the first of which is described in Lemma~\ref{lem-deloop} and the other in Proposition~\ref{square-iso-prop}. Both Lemma~\ref{lem-deloop} and Propostion~\ref{lem-GE} are due to Bar-Natan \cite{FastKh}.

\begin{lemma}\label{lem-deloop}
If an object $S$ in $\Cobbl$ contains a closed loop $l$, it is isomorphic in $\Mat(\Cobbl)$ to the direct sum of two copies $S'\{-1\}$ and $S'\{+1\}$ of $S$ in which $l$ is removed, one with a degree shift of $-1$ and the other with a degree shift of $+1$. Symbolically, $\circleprop \cong \emptyset \{-1\} \oplus \emptyset \{+1\}$.
\end{lemma}

\begin{proof}
The isomorphisms are as follows.
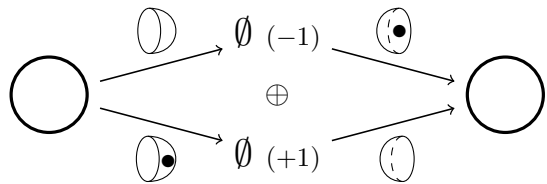
\[
	\begin{tikzpicture}[auto,semithick]
\node (l) at (-3,0) {\tikz \draw[very thick] (0,0) circle (0.5cm);};
\node (r) at (3,0) {\tikz \draw[very thick] (0,0) circle (0.5cm);};
\node (0a) at (0,0.8) {\large $\emptyset$ \footnotesize $(-1)$};
\node (0b) at (0,-0.8) {\large $\emptyset$ \footnotesize $(+1)$};
\node (0) at (0,0) {$\oplus$};

\draw[->] (l) to node[anchor=south] {\begin{tikzpicture}\begin{scope}[rotate around={90:(0,0)},-] %TL
\draw (0,0) ellipse (0.3cm and 0.15cm);
\draw [out=-90,in=-90,looseness=2] (-0.3,0) to (0.3,0);
\end{scope}\end{tikzpicture}} (0a);

\draw[->] (l) to node[anchor=north] {\begin{tikzpicture}\begin{scope}[rotate around={(90:(0,0))},-] %BL
\draw (0,0) ellipse (0.3cm and 0.15cm);
\draw [out=-90,in=-90,looseness=2.1] (-0.3,0) to (0.3,0);
\node (bt) at (.2,-0.27) {$\bullet$}; %weird bullet placement, 0 -> 0.2
\end{scope}\end{tikzpicture}} (0b);

\draw[->] (0a) to node[anchor=south] {\begin{tikzpicture}\begin{scope}[rotate around={90:(0,0)},-] %TR
\draw[dashed] (0.3,0) arc(0:180:0.3cm and 0.15cm);
\draw (0.3,0) arc(0:-180:0.3cm and 0.15cm);
\draw [out=90,in=90,looseness=1.7] (-0.3,0) to (0.3,0);
\node (bt) at (-0.25,0) {$\bullet$};\end{scope}\end{tikzpicture}} (r);

\draw[->] (0b) to node[anchor=north] {\begin{tikzpicture}\begin{scope}[rotate around={90:(0,0)},-] %BR
\draw[dashed] (0.3,0) arc(0:180:0.3cm and 0.15cm);
\draw (0.3,0) arc(0:-180:0.3cm and 0.15cm);
\draw [out=90,in=90,looseness=1.7] (-0.3,0) to (0.3,0);
\end{scope}\end{tikzpicture}} (r);
\end{tikzpicture}
\]
\end{proof}

Although this `delooping' process actually increases the number of subobjects in the complex, in practice many of the maps between the subobjects become isomorphisms. In such a case, the categorical version of Gaussian elimination below allows us to discard pairs of subobjects related via an isomorphism. Like the previously lemma, we use Gaussian elimination extensively in what follows.

\begin{lemma}\label{lem-GE}
If $\phi:b_1 \rightarrow b_2$ is an isomorphism in some additive category $\Cc$, then the four term complex in $\Mat(\Cc)$
\[
	\begin{tikzpicture}[auto,semithick]
	\node (0) at (-2,0) {$\cdots$};
	\node (1) at (0,0) {$\begin{bmatrix} C \end{bmatrix}$};
	\node (2) at (3,0) {$\begin{bmatrix} b_1 \\ D \end{bmatrix}$};
	\node (3) at (6,0) {$\begin{bmatrix} b_2 \\ E \end{bmatrix}$};
	\node (4) at (9,0) {$\begin{bmatrix} F \end{bmatrix}$};
	\node (5) at (11,0) {$\cdots$};
	\draw[->] (0) to (1);
	\draw[->] (1) to node {$\begin{pmatrix} \alpha \\ \beta	\end{pmatrix}$} (2);
	\draw[->] (2) to node {$\begin{pmatrix} \phi & \delta \\ \gamma & \varepsilon \end{pmatrix}$} (3);
	\draw[->] (3) to node {$\begin{pmatrix} \mu & \nu \end{pmatrix}$} (4);
	\draw[->] (4) to (5);
\end{tikzpicture}
\]
is homotopy equivalent to the complex
\[
	\begin{tikzpicture}[auto,semithick]
	\node (0) at (-1.5,0) {$\cdots$};
	\node (1) at (0,0) {$\begin{bmatrix} C \end{bmatrix}$};
	\node (2) at (3,0) {$\begin{bmatrix} D \end{bmatrix}$};
	\node (3) at (7.5,0) {$\begin{bmatrix} E \end{bmatrix}$};
	\node (4) at (10.5,0) {$\begin{bmatrix} F \end{bmatrix}$};
	\node (5) at (12,0) {$\cdots.$};
	\draw[->] (0) to (1);
	\draw[->] (1) to node {$ \begin{pmatrix} \beta \end{pmatrix} $} (2);
	\draw[->] (2) to node {$ \begin{pmatrix} \varepsilon - \gamma \phi\inv \delta \end{pmatrix} $} (3);
	\draw[->] (3) to node {$ \begin{pmatrix} \nu \end{pmatrix} $} (4);
	\draw[->] (4) to (5);
\end{tikzpicture}
\]
Here $C,D,F$ and $F$ are arbitrary columns of objects in $\Cc$ and all Greek letters (other than $\phi$) represent arbitrary matrices of morphisms in $\Cc$; all matrices appearing in these complexes are block-matrices with blocks as specified.
\end{lemma}

\section{Khovanov complexes of rational tangles}\label{sec-KhRts}
In Section~\ref{sec-RTs} we recalled that every positive or negative rational tangle is isotopic to a canonical form $\la a_1,a_2,\ldots,a_n \ra$ where the $a_i$ are all non-negative or all non-positive respectively (Corollary~\ref{coro-canonical-form}). Such a rational tangle can be constructed from a finite sequence of additions and products with the $[+1]$ or $[-1]$ tangle respectively. This together with the locality of the Khovanov bracket means that the Khovanov complex of a rational tangle can be constructed inductively from $\Kb{\pm 1}$ by a sequence of intermediate complexes, where each is obtained from the next by adding a crossing then immediately simplifying the resulting complex. The main result in this section explicitly shows how, at each step, the complexes simplify to zig-zag complexes.

For the rest of the note, we'll work with positive tangles. The results for the case of negative tangles are completely analogous. We use the notation $\Kb{a_1,a_2,\ldots,a_n}$ to denote the Khovanov complex of the rational tangle $\la a_1,\ldots,a_n\ra$.

We first calculate the Khovanov complex of the $n$-twist. Calculations like this one have been known for a long time (c.f. the $sl_3$ and $sl_n$ analogies in \cite{Scottsl3} and \cite{Danielsln}), but we spell out the details here since it is a crucial component of the later arguments.

\begin{prop}
\label{proof-ints}
Let $n>0$ and $[n]$ have orientation type II (Section~\ref{sec-RTs}). Then $\Kbp{n}$ is homotopy equivalent to
\[	
	\begin{tikzpicture}[auto,semithick, scale=1.2]
\node (1) at (-2,0) {\nodeinf};
\draw[thick] (-2.4,-0.4) -- (-1.6,-0.4);
\node (1a) at (-2,-0.65) {\footnotesize $(-3n+1)$};
\node (2) at (0,0) {$\cdots$};
\node (3) at (2.5,0) {\nodeinf};
\node (3a) at (2.5,-0.65) {\footnotesize $(-n-5)$};
\node (4) at (5,0) {\nodeinf};
\node (4a) at (5,-0.65) {\footnotesize $(-n-3)$};
\node (5) at (7.5,0) {\nodeinf};
\node (5a) at (7.5,-0.65) {\footnotesize $(-n-1)$};
\node (6) at (9.5,0) {\nodezero\;.};
\node (6a) at (9.5,-0.65) {\footnotesize $(-n)$};

\draw[->] (1) to node {} (2);
\draw[->] (2) to node {$\vcenter{\hbox{\nodeinfdota}} - \vcenter{\hbox{\nodeinfdotb}}$} (3);
\draw[->] (3) to node {$\vcenter{\hbox{\nodeinfdota}} + \vcenter{\hbox{\nodeinfdotb}}$} (4);
\draw[->] (4) to node {$\vcenter{\hbox{\nodeinfdota}} - \vcenter{\hbox{\nodeinfdotb}}$} (5);
\draw[->] (5) to node {\nodesaddleh} (6);

\end{tikzpicture}
\]
\end{prop}

\begin{proof}

The case $n=1$ follows directly from the definition of the Khovanov bracket. The case $n=2$ is similar to the proof of the invariance of the Khovanov bracket under R2. Namely, we write $[2] = [1] + [1]$ and construct the planar arc diagram $D$ corresponding to tangle addition (Figure~\ref{kh-2}). Since the Khovanov bracket is a planar algebra morphism, $\Kb{2} = \Kb{D([1],[1])} = D(\Kb{1},\Kb{1})$. The complex $D(\Kb{1},\Kb{1})$ is constructed in Figure~\ref{kh-2} below.
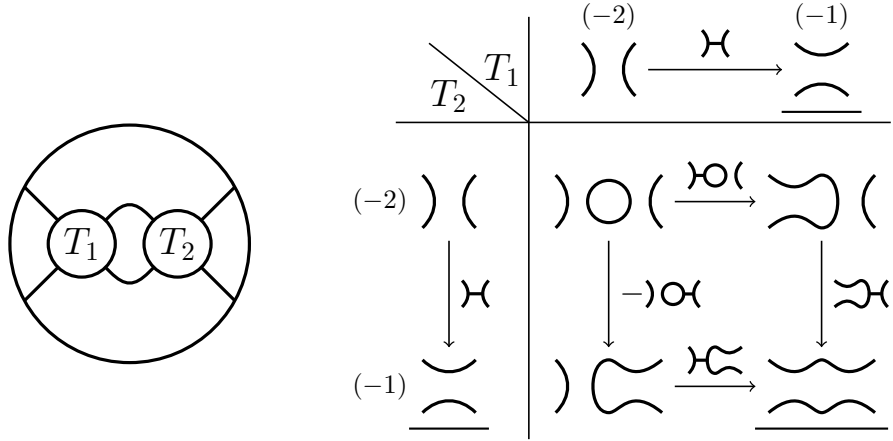
\begin{figure}[ht!]
	\centering
	\begin{tikzpicture}[auto,semithick,scale=.7]

\begin{scope}[very thick, scale=0.9,yshift=3cm]
\draw (0,0) circle (2.5cm);
\draw (-1,0) circle (.7cm);
\draw (1,0) circle (.7cm);
\node (T1) at (-1,0) {\large $T_1$};
\node (T2) at (1,0) {\large $T_2$};

\draw[xshift=-1cm] (45:.7) [xshift=0.1cm] -- +(0,.1) [xshift=0.9cm] -- (-0.4,0.6) .. controls (-0.1,.9) and (0.1,0.9) .. (0.4,0.6) [xshift=1cm] -- (135:0.7);
\draw[xshift=-1cm] (-45:.7) [xshift=0.1cm] -- +(0,-.1) [xshift=0.9cm] -- (-0.4,-0.6) .. controls (-0.1,-.9) and (.1,-.9) .. (0.4,-.6) [xshift=1cm] -- (-135:.7);
\draw[xshift=-1cm] (135:0.7) --  ++(135:1);
\draw[xshift=-1cm] (-135:0.7) --  ++(-135:1);
\draw[xshift=1cm] (45:0.7) --  ++(45:1);
\draw[xshift=1cm] (-45:0.7) --  ++(-45:1);
\end{scope}

\begin{scope}[xshift=9cm]
\draw (-1.5,-1) -- (-1.5,7);
\draw (-4,5) -- (5.3,5);
\draw (-1.5,5) -- ++(-1.875,1.5);

\node (t1l) at (-2,6) {\large $T_1$};
\node (t2l) at (-3,5.5) {\large $T_2$};

\draw[thick] (3.25,5.2) -- (4.75,5.2);
\draw[thick] (-3.75,-0.8) -- (-2.25,-0.8);
\draw[thick] (2.75,-0.8) -- (5.25,-0.8);

\node (00) at (0,0) {\nodemhinfzero};
\node (01) at (0,3.5) {\nodeclownh};
\node (10) at (4,0) {\nodemhzerozero};
\node (11) at (4,3.5) {\nodemhzeroinf};

\node (u0) at (0,6) {\nodeinf};
\node (u1) at (4,6) {\nodezero};
\node (u0a) at (0,7) {\footnotesize $(-2)$};
\node (u0a) at (4,7) {\footnotesize $(-1)$};

\node (l1) at (-3,3.5) {\nodeinf};
\node (l0) at (-3,0) {\nodezero};
\node (l0a) at (-4.3,3.5) {\footnotesize $(-2)$};
\node (l1a) at (-4.3,0) {\footnotesize $(-1)$};

\draw[->] (01) to node {\nodeclownsadl} (11);
\draw[->] (01) to node {$-\vcenter{\hbox{\nodeclownsadr}}$} (00);
\draw[->] (00) to node {\nodeinfzerosad} (10);
\draw[->] (11) to node {\nodezeroinfsad} (10);

\draw[->] (u0) to node {\nodesaddleh} (u1);
\draw[->] (l1) to node {\nodesaddleh} (l0);
\end{scope}

\end{tikzpicture}
	\caption{\textsc{left} The integer tangle $[2]$ is the result of placing the $[1]$ tangle in both holes of the planar arc diagram illustrated. \textsc{right} Calculating $\Kb{2}$ from $\Kb{1}$ and the planar arc diagram to the left.}
	\label{kh-2}
\end{figure}
We simplify the complex as follows. Delooping the object in the NW corner gives us
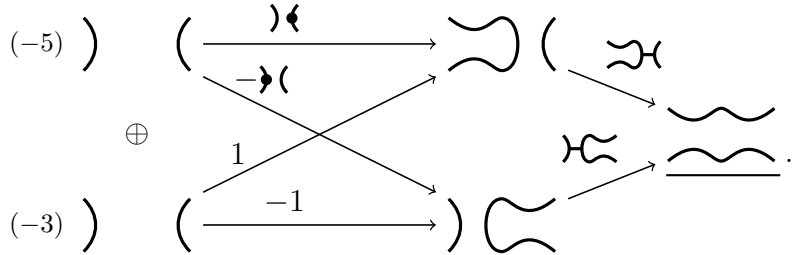
\[
	\begin{tikzpicture}[auto,semithick,scale=1.2]
\node (01) at (0,2) {\nodeclownhdeloop};
\node (0c) at (0,1) {$\oplus$};
\node (01a) at (-1.1,2) {\footnotesize $(-5)$};
\node (00) at (0,0) {\nodeclownhdeloop};
\node (00a) at (-1.1,0) {\footnotesize $(-3)$};
\node (10) at (4,0) {\nodemhinfzero};
\node (11) at (4,2) {\nodemhzeroinf};
\node (3) at (6.5,1) {\nodemhzerozero \;.};

\draw[->] (01) to node[pos=0.35] {\nodeinfdotb} (11);
\draw[->] (00) to node[pos=0.35] {$-1$} (10);
\draw[->] (01) to node[pos=0.25,anchor=south] {$-\vcenter{\hbox{\nodeinfdota}}$} (10);
\draw[->] (00) to node[pos=0.15,anchor=south] {$1$} (11);
\draw[->] (10) to node[pos=0.7] {\nodeinfzerosad} (3);
\draw[->] (11) to node[pos=0.3] {\nodezeroinfsad} (3);

\draw[thick] (5.8,0.55) -- (7.05,0.55);

\end{tikzpicture}
\]
After applying Gaussian elimination and negating the first map, this simplifies to
\[
	\begin{tikzpicture}[auto,semithick, scale=1.2]
\node (1) at (-0.5,0) {\nodeinf};
\node (1a) at (-0.5,-0.65) {\footnotesize $(-5)$};
\node (2) at (2,0) {\nodeinf};
\node (2a) at (2,-0.65) {\footnotesize $(-3)$};
\node (3) at (4,0) {\nodezero\;.};
\draw[thick] (3.5,-0.4) -- (4.3,-0.4);
\node (3a) at (4,-0.65) {\footnotesize $(-2)$};

\draw[->] (1) to node {$\vcenter{\hbox{\nodeinfdota}} - \vcenter{\hbox{\nodeinfdotb}}$} (2);
\draw[->] (2) to node {\nodesaddleh} (3);
\end{tikzpicture}
\]
Now assume the claim is true for some $n\geq 2$. By using the same `addition' planar arc diagram in Figure~\ref{kh-2} above with $T_1 = [n]$ and $T_2 = [1]$, we have $\Kb{n+1} = D(\Kb{n},\Kb{1})$. This complex is in Figure~\ref{nplus1complex}, and simplifies as follows.
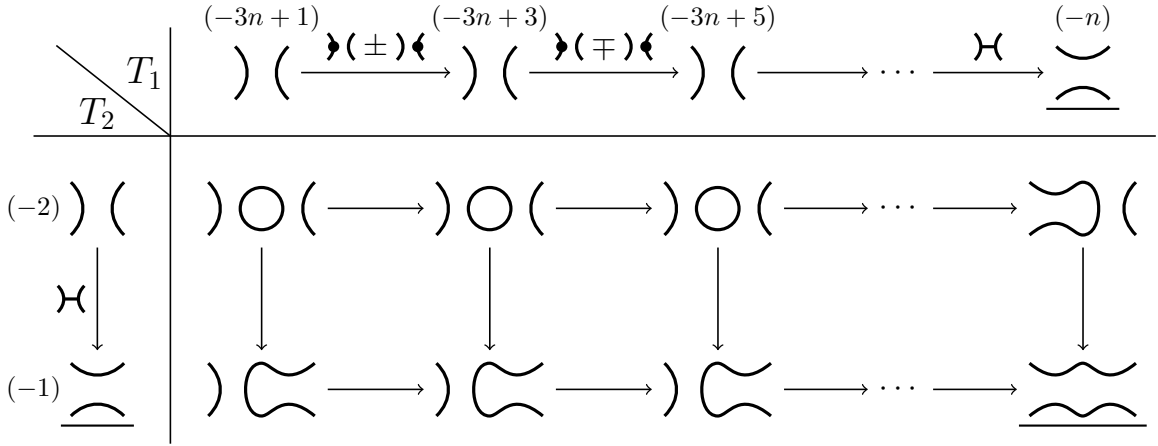
\begin{figure}[ht!]
	\centering
	\begin{tikzpicture}[auto,semithick,scale=1.2]
\draw (-1,-0.6) -- (-1,4);
\draw (-2.5,2.8) -- (9.8,2.8);
%\draw (-1,2.8) -- (-2.5,4);
\draw (-1,2.8) -- ++ (-1.25,1);
\draw[thick] (-2.2,-0.4) -- (-1.4,-0.4);
\draw[thick] (8.6,3.1) -- (9.4,3.1);
\draw[thick] (8.3,-0.4) -- (9.7,-0.4);

\node (t1l) at (-1.28,3.5) {\large $T_1$};
\node (t2l) at (-1.8,3.05) {\large $T_2$};

\node (1) at (0,2) {\nodeclownh};
\node (2) at (2.5,2) {\nodeclownh};
\node (3) at (5,2) {\nodeclownh};
\node (4) at (7,2) {$\cdots$};
\node (5) at (9,2) {\nodemhzeroinf};

\node (6) at (0,0) {\nodemhinfzero};
\node (7) at (2.5,0) {\nodemhinfzero};
\node (8) at (5,0) {\nodemhinfzero};
\node (9) at (7,0) {$\cdots$};
\node (10) at (9,0) {\nodemhzerozero};

\node (u1) at (0,3.5) {\nodeinf};
\node (u1a) at (0,4.1) {\footnotesize $(-3n + 1)$};
\node (u2) at (2.5,3.5) {\nodeinf};
\node (u2a) at (2.5,4.1) {\footnotesize $(-3n + 3)$};
\node (u3) at (5,3.5) {\nodeinf};
\node (u3a) at (5,4.1) {\footnotesize $(-3n + 5)$};
\node (u4) at (7,3.5) {$\cdots$};
\node (u5) at (9,3.5) {\nodezero};
\node (u5a) at (9,4.1) {\footnotesize $(-n)$};

\node (u6) at (-1.8,2) {\nodeinf};
\node (u6a) at (-2.5,2) {\footnotesize $(-2)$};
\node (v6) at (-1.8,0) {\nodezero};
\node (u6a) at (-2.5,0) {\footnotesize $(-1)$};

\draw[->] (1) to node {} (2);
\draw[->] (2) to node {} (3);
\draw[->] (3) to node {} (4);
\draw[->] (4) to node {} (5);

\draw[->] (1) to node {} (6);
\draw[->] (2) to node {} (7);
\draw[->] (3) to node {} (8);
\draw[->] (5) to node {} (10);

\draw[->] (6) to node {} (7);
\draw[->] (7) to node {} (8);
\draw[->] (8) to node {} (9);
\draw[->] (9) to node {} (10);

\draw[->] (u1) to node {$\vcenter{\hbox{\nodeinfdota}} \pm \vcenter{\hbox{\nodeinfdotb}}$} (u2);
\draw[->] (u2) to node {$\vcenter{\hbox{\nodeinfdota}} \mp \vcenter{\hbox{\nodeinfdotb}}$} (u3);
\draw[->] (u3) to node {} (u4);
\draw[->] (u4) to node {\nodesaddleh} (u5);
\draw[->] (u6) to node[swap] {\nodesaddleh} (v6);
\end{tikzpicture}
	\caption{The complex $\Kb{n+1}$ can be computed from $\Kb{n}$ and $\Kb{1}$ using the same planar arc diagram in Figure~\ref{kh-2}. For readability the morphisms and quantum gradings of the subobjects have been omitted.}
	\label{nplus1complex}
\end{figure}

Delooping the NW corner of the complex and then applying Gaussian elimination gives us the following complex.
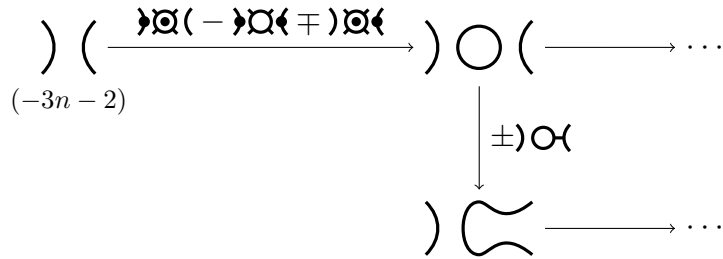
\[
	\begin{tikzpicture}[auto,scale=1.2]
\node (1) at (-1.5,2) {\nodeinf};
\node (1a) at (-1.5,1.4) {\footnotesize $(-3n - 2)$};
\node (2) at (3,2) {\nodeclownh};
\node (3) at (5.5,2) {$\cdots$};
\node (4) at (3,0) {\nodemhinfzero};
\node (5) at (5.5,0) {$\cdots$};

\draw[->] (1) to node {$\vcenter{\hbox{\nodeclownhspab}} - \vcenter{\hbox{\nodeclownhspac}} \mp \vcenter{\hbox{\nodeclownhspbc}}$} (2);
\draw[->] (2) to node {$\pm\vcenter{\hbox{\nodeclownsadr}}$} (4);
\draw[->] (2) to node {} (3);
\draw[->] (4) to node {} (5);

\end{tikzpicture}
\]
The morphism out of the subobject in the NW corner may appear complicated, but simplifies when the next term is delooped.

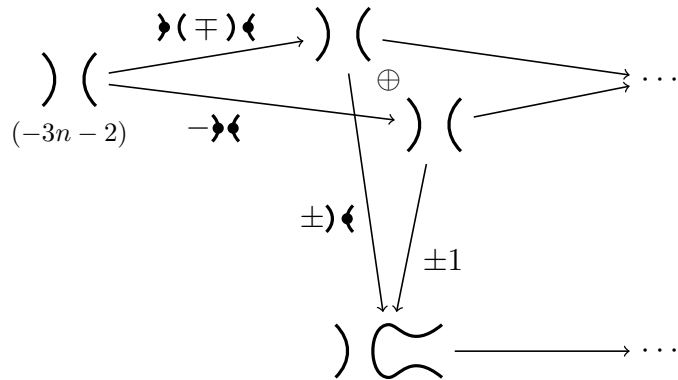
\[
	\begin{tikzpicture}[auto,semithick,scale=1.2]
\node (1) at (-0.5,3) {\nodeinf};
\node (1a) at (-0.5,2.4) {\footnotesize $(-3n-2)$};
\node (2s) at (3,3) {$\oplus$};
\node (2a) at (2.5,3.5) {\nodeinf};
\node (2b) at (3.5,2.5) {\nodeinf};
\node (3) at (6,3) {$\cdots$};
\node (4) at (3,0) {\nodemhinfzero};
\node (5) at (6,0) {$\cdots$};

\draw[->] (1) to node[above=0.05cm] {$\vcenter{\hbox{\nodeinfdota}} \mp \vcenter{\hbox{\nodeinfdotb}}$} (2a);
\draw[->] (1) to node[swap] {$-\vcenter{\hbox{\nodeinfdotab}}$} (2b);
\draw[->] (2a) to node {} (3);
\draw[->] (2b) to node {} (3);
\draw[->] (2a) to node[swap] {$\pm \vcenter{\hbox{\nodeinfdotb}} $} (4);
\draw[->] (2b) to node {$\pm 1$} (4);
\draw[->] (4) to node {} (5);
\end{tikzpicture}
\]

A further application of Gaussian elimination removes two more subobjects.

We continue to move left to right through the complex using this method -- delooping each successive tangle in the north row and eliminating the subobject south of it. The process `conjugates' the maps between the northern objects. By this, we mean the maps $D(\vcenter{\hbox{\nodeinfdota}} \pm \vcenter{\hbox{\nodeinfdotb}}, 1)$ simplify to $\vcenter{\hbox{\nodeinfdota}} \mp \vcenter{\hbox{\nodeinfdotb}}$.

After working our way to the right end of the chain complex, eventually only two subobjects on the southern row remain. These subobjects, together with those directly north of them, form a square consisting only of saddle maps. This square is the complex associated to $\Kb{2}$. The rest of the complex, consisting of the terms we have delooped is attached to this square like a tail. The square simplifies in exactly the same way as we did previously for the $[2]$ tangle in the $n=2$ case, from which the proposition follows.
\end{proof}

Rational tangles in general do not have complexes as elementary as these. However, it is true that any non-zero map between two indecomposable subobjects in the minimal complex of a rational tangle is (up to sign) one of the six morphisms below.

\[
	\begin{tikzpicture}[scale=1.2]
\node (1) at (0,1) {$a = \vcenter{\hbox{\begin{tikzpicture}[scale=0.5]\tanginfdota\end{tikzpicture}}} + \vcenter{\hbox{\begin{tikzpicture}[scale=0.5]\tanginfdotb\end{tikzpicture}}}$};
\node (2) at (3,1) {$c = \vcenter{\hbox{\begin{tikzpicture}[scale=0.5]\tangzerodota\end{tikzpicture}}} + \vcenter{\hbox{\begin{tikzpicture}[scale=0.5]\tangzerodotb\end{tikzpicture}}}$};
\node (3) at (6,0.5) {$s = \vcenter{\hbox{\begin{tikzpicture}[scale=0.5]\saddleh\end{tikzpicture}}}, \;\; \vcenter{\hbox{\begin{tikzpicture}[scale=0.5]\saddlev\end{tikzpicture}}}$};
\node (4) at (0,0) {$b = \vcenter{\hbox{\begin{tikzpicture}[scale=0.5]\tanginfdota\end{tikzpicture}}} - \vcenter{\hbox{\begin{tikzpicture}[scale=0.5]\tanginfdotb\end{tikzpicture}}}$};
\node (5) at (3,0) {$d = \vcenter{\hbox{\begin{tikzpicture}[scale=0.5]\tangzerodota\end{tikzpicture}}} - \vcenter{\hbox{\begin{tikzpicture}[scale=0.5]\tangzerodotb\end{tikzpicture}}}$};
\end{tikzpicture}
\]

This property is a consequence of Theorem~\ref{thm-main}, and we will often use these abbreviations in the sequel to minimize clutter when depicting complexes.

We will also not worry about the signs of morphisms between subobjects in general. Negating a differential in a chain complex gives an isomorphic complex, and we are only considering complexes up to homotopy. It should be noted that we cannot just negate any map between subobjects though: this would in general not even give a complex. The fact that the squares in our complex anti-commute means we still need to keep track of the relative sign of adjacent maps though.

We now come to a type of isomorphism we call a `square isomorphism'. If there is any piece of machinery to take away from the article, this is it, since the entire article is based on it.

\begin{prop}\label{square-iso-prop}
	The vertical maps in Figure~\ref{fig-square_iso} constitute an isomorphism.
\end{prop}
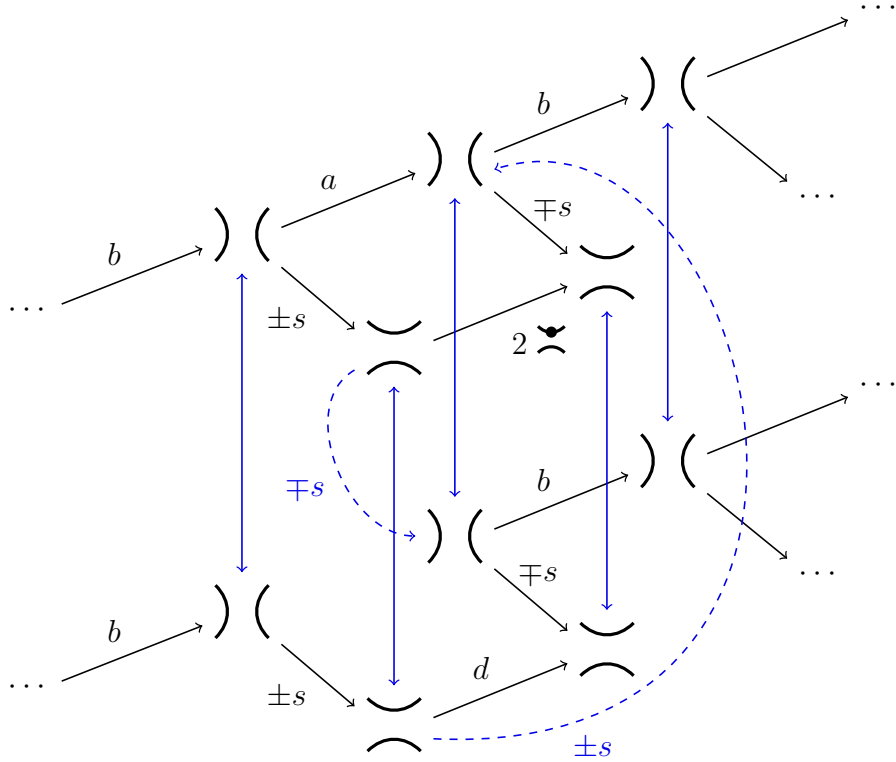
\begin{figure}[ht]
	\centering
	\begin{tikzpicture}[auto,semithick,scale=1]
\node (bld) at (-2.8,-1) {$\cdots$};
\node (000) at (0,0) {\nodeinf};
\node (001) at (2,-1.5) {\nodezero};
\node (010) at (2.8,1) {\nodeinf};
\node (011) at (4.8,-0.5) {\nodezero};
\node (brd) at (5.6,2) {\nodeinf};
\node (brdu) at (8.4,3) {$\cdots$};
\node (brdd) at (7.6,0.5) {$\cdots$};

\draw[->] (bld.10) to node {$b$} (000.200);
\draw[->] (000.-40) to node[swap] {$\pm s$} (001.155);
% \draw[->] (000.10) to node {} (010.200);
\draw[->] (010.-40) to node[pos=0.6,above=0.15cm] {$\mp s$} (011.155);
\draw[->] (001.10) to node {$d$} (011.200);
\draw[->] (010.10) to node {$b$} (brd.200);
\draw[->] (brd.10) to node {} (brdu.200);
\draw[->] (brd.-40) to node {} (brdd.155);

\begin{scope}[yshift=5cm]
\node (tld) at (-2.8,-1) {$\cdots$};
\node (100) at (0,0) {\nodeinf};
\node (101) at (2,-1.5) {\nodezero};
\node (110) at (2.8,1) {\nodeinf};
\node (111) at (4.8,-0.5) {\nodezero};
\node (trd) at (5.6,2) {\nodeinf};
\node (trdu) at (8.4,3) {$\cdots$};
\node (trdd) at (7.6,0.5) {$\cdots$};

\draw[->] (tld.10) to node {$b$} (100.200);
\draw[->] (100.-40) to node[swap] {$\pm s$} (101.155);
\draw[->] (100.10) to node {$a$} (110.200);
\draw[->] (110.-40) to node[pos=0.8,above=0.15cm] {$\mp s$} (111.155);
\draw[->] (101.10) to node[swap] {$2\;\nodezerodota$} (111.200);
\draw[->] (110.10) to node {$b$} (trd.200);
\draw[->] (trd.10) to node {} (trdu.200);
\draw[->] (trd.-40) to node {} (trdd.155);
\end{scope}

\draw[blue!90!black,<->] (000) to node {} (100);
\draw[blue!90!black,<->] (001) to node {} (101);
\draw[blue!90!black,<->] (010) to node {} (110);
\draw[blue!90!black,<->] (011) to node {} (111);
\draw[blue!90!black,<->] (trd) to node {} (brd);

\draw[dashed,blue!90!black,->] (101) to [out=210,in=180] node[swap] {$\mp s$} (010);
\draw[pos=0.1,dashed,blue!90!black,->] (001.-20) .. controls (9,-2) and (6.7,7) .. node[swap] {$\pm s$} (110.-15);
\end{tikzpicture}
	\caption{A square isomorphism. Horizontal layers are the isomorphic complexes, the vertical maps (colored blue) the actual isomorphism. Straight arrows between layers are just identities, the curved arrows more interesting.}
	\label{fig-square_iso}
\end{figure}
\begin{proof}
To actually check this is an isomorphism, we need to verify three points.
\begin{enumerate}
	\item The top and bottom layers are actually chain complexes.
	\item The collection of maps constitutes a chain map.
	\item The chain maps are inverses of one another.
\end{enumerate}

The first is easy -- as we see in an example shortly, the top layer forms part of an actual Khovanov complex, so we only need to check that the squares in the bottom layer anticommute. Only one of the squares in this is different from the top layer, so we need only check that $d \circ (\pm s) = -(\mp s) \circ 0 = 0$. This is true since the saddle in $d \circ \pm s$ connects both sheets in each of the cobordisms constituting $d$; the components then cancel.

The second point amounts to showing that several compositions commute (Figure~\ref{fig-square_iso-ugly}). This involves a few calculations, but is not difficult. (Hint: apply neck-cutting.)

The third point is the easiest of them all as we merely need to check that
\[
	\begin{pmatrix}
	\begin{array}{cc}
		1 & \mp s \\
		0 & 1 \\
	\end{array}
	\end{pmatrix}^{-1} = \begin{pmatrix}
					\begin{array}{cc}
						1 & \pm s \\
						0 & 1\\
					\end{array}
					\end{pmatrix}.
\]
\end{proof}

\begin{figure}[ht]
	\centering
	\begin{tikzpicture}[semithick,auto,scale=1]
\node (1) at (-1.5,6) {\nodeinf};
\node (2u) at (2,6.8) {\nodeinf};
\node (2) at (2,6) {$\oplus$};
\node (2d) at (2,5.2) {\nodezero};
\node (3) at (6,6) {$\oplus$};
\node (3u) at (6,6.8) {\nodeinf};
\node (3d) at (6,5.2) {\nodezero};
\begin{scope}[yshift=-5cm]
\node (4) at (-1.5,6) {\nodeinf};
\node (5) at (2,6) {$\oplus$};
\node (5u) at (2,6.8) {\nodeinf};
\node (5d) at (2,5.2) {\nodezero};
\node (6) at (6,6) {$\oplus$};
\node (6u) at (6,6.8) {\nodeinf};
\node (6d) at (6,5.2) {\nodezero};
\end{scope}

\draw[->] (-1,6) to node {$\begin{pmatrix} a \\ \pm s \end{pmatrix}$} (1.5,6);
\draw[->] (2.5,6) to node {$\begin{pmatrix} b & 0 \\ \mp s & 2\;\nodezerodota \end{pmatrix}$} (5.5,6);
\draw[->] (-1,1) to node {$\begin{pmatrix} 0 \\ \pm s \end{pmatrix}$} (1.5,1);
\draw[->] (5) to node {$\begin{pmatrix} b & 0 \\ \mp s & d \end{pmatrix}$} (5.5,1);

\draw[blue!90!black,<->] (1) to node {$1$} (4);
\draw[blue!90!black,->] (2d.290) to node {$\begin{pmatrix} 1 & \mp s \\ 0 & 1 \end{pmatrix}$} (5u.70);
\draw[blue!90!black,<-] (2d.250) to node[swap] {$\begin{pmatrix} 1 & \pm s \\ 0 & 1 \end{pmatrix}$} (5u.110);
\draw[blue!90!black,<->] (3d) to node {$1$} (6u);
\end{tikzpicture}
	\caption{A square isomorphism, in a more traditional form.}
	\label{fig-square_iso-ugly}
\end{figure}

We will illustrate why this `square isomorphism' is useful with an example.

\begin{example}\label{eg-5-1}
Let us calculate $\Kb{5,1}$ without worrying about the quantum gradings or the precise homological degrees (but we will of course keep track of the relative homological degrees). With $[5]$ as $T_1$ and $[1]$ as $T_2$ in the planar arc diagram corresponding to the product of two tangles (analogous to that in Figure~\ref{kh-2}), we obtain the following complex.

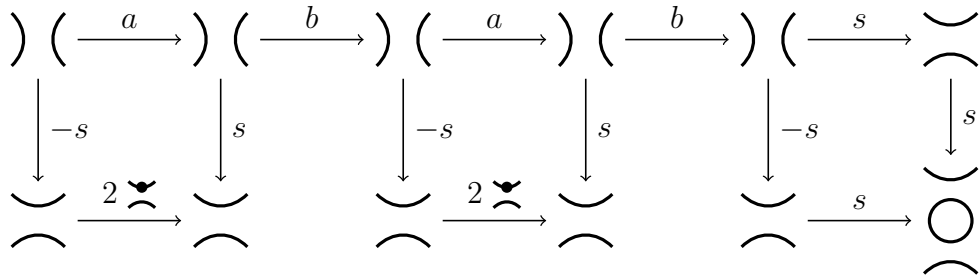
\[
	\begin{tikzpicture}[auto,semithick,scale=1.2]
\node (1) at (0,2) {\nodeinf};
\node (2) at (2,2) {\nodeinf};
\node (3) at (4,2) {\nodeinf};
\node (4) at (6,2) {\nodeinf};
\node (5) at (8,2) {\nodeinf};
\node (6) at (10,2) {\nodezero};
\node (7) at (0,0) {\nodezero};
\node (8) at (2,0) {\nodezero};
\node (9) at (4,0) {\nodezero};
\node (10) at (6,0) {\nodezero};
\node (11) at (8,0) {\nodezero};
\node (12) at (10,0) {\nodeclownv};

\draw[->] (1) to node {$a$} (2);
\draw[->] (2) to node {$b$} (3);
\draw[->] (3) to node {$a$} (4);
\draw[->] (4) to node {$b$} (5);
\draw[->] (5) to node {$s$} (6);

\draw[->] (1) to node {$-s$} (7);
\draw[->] (2) to node {$s$} (8);
\draw[->] (3) to node {$-s$} (9);
\draw[->] (4) to node {$s$} (10);
\draw[->] (5) to node {$-s$} (11);
\draw[->] (6) to node {$s$} (12);

\draw[->] (7) to node {$2\;\vcenter{\hbox{\nodezerodota}}$} (8);
\draw[->] (9) to node {$2\; \vcenter{\hbox{\nodezerodota}}$} (10);
\draw[->] (11) to node {$s$} (12);

\end{tikzpicture}
\]

The right-most anticommuting square of the complex, consisting only of saddles, simplifies in the same way as in the calculation of $\Kb{2}$ in the proof of Proposition~\ref{proof-ints}. After using the square isomorphism on the left-two anti-commuting squares, and some isomorphisms to remove minus signs, we obtain the following complex.
\[
	\begin{tikzpicture}[auto,semithick,scale=1.2]
\node (1) at (0,2) {\nodeinf};
\node (2) at (2,2) {\nodeinf};
\node (3) at (4,2) {\nodeinf};
\node (4) at (6,2) {\nodeinf};
\node (5) at (8,2) {\nodeinf};
\node (6) at (10,2) {\nodezero};
\node (7) at (0,0) {\nodezero};
\node (8) at (2,0) {\nodezero};
\node (9) at (4,0) {\nodezero};
\node (10) at (6,0) {\nodezero};
\node (12) at (10,0) {\nodezero};

\draw[->] (1) to node {$s$} (7);
\draw[->] (7) to node {$d$} (8);
\draw[->] (2) to node {$s$} (8);
\draw[->] (2) to node {$b$} (3);
\draw[->] (3) to node {$s$} (9);
\draw[->] (9) to node {$d$} (10);
\draw[->] (4) to node {$s$} (10);
\draw[->] (4) to node {$b$} (5);
\draw[->] (5) to node {$s$} (6);
\draw[->] (6) to node {$d$} (12);
\end{tikzpicture}
\]

So $\Kb{5,1}$ is actually homotopic to a zig-zag complex. Representing the complex this way, however, makes it hard to determine the geometry of underlying zig-zag, as well as being inefficient. We instead represent this as a `dot diagram', illustrated below and explained in the caption.

\begin{figure}[ht!]
	\centering
	\begin{tikzpicture}[very thick,scale=.7]
\foreach \x in {0,...,5}
{
	\draw[gray,xshift=0.5cm] (\x,-0.5) -- (\x,4.5);
}
\draw (0,0) -- (2,0.8) -- (1,1.2) -- (4,2.4) -- (3,2.8) -- (6,4);
\foreach \p in {(0,0),(1,1.2),(2,1.6),(3,2.8),(4,3.2)}
{\filldraw[fill=white,very thick] \p circle (4.5pt);}
\foreach \p in {(1,0.4),(2,0.8),(3,2),(4,2.4),(5,3.6),(6,4)}
{\filldraw \p circle (2pt);}
\end{tikzpicture}
\caption{The dot diagram of $\Kb{5,1}$. Vertical lines separate regions of the same homological height. $[\infty]$ tangles are represented by circles, $[0]$ tangles by dots. Black lines connecting the circles and dots indicate non-zero morphisms between subobjects.}
\end{figure}

We haven't labeled the morphisms, but there is no need to since all can be determined from the diagram --- each of the five `straight' segments in the dot diagram contain a saddle, which determines the rest of the morphisms along the segment. (Recall that, as remarked earlier, in the minimal complex of a rational tangle, there are (up to sign) only six morphisms. This and the fact that $\partial^2 = 0$ determines the other morphisms. For example, if a saddle map $\miniinf \rightarrow \minizero$ is followed by an unknown morphism $\minizero \rightarrow \minizero$, the unknown morphism must be ether $c$ or $d$, but only $d$ composes with $s$ to give 0. Similar statements apply to the other unknown compositions. This concludes the calculation.
\end{example}

For the remainder of the paper, we'll describe a zig-zag complex by its backbone of non-zero morphisms, writing these as a word. Staring at one $z$-end of the zig-zag complex, we follow the backbone of non-zero morphisms until we reach the other $z$-end, writing down the morphisms along the way. We call this word the \emph{morphism string} of a zig-zag complex. For example the morphism string of the minimal complex for $\Kb{5}$ is $ababs$.

Strictly speaking, the morphism string of a zig-zag complex isn't well defined -- it depends on which $z$-end one starts at -- but we'll adopt the convention that we begin at the $z$-end with the lowest homological height, if the two differ. (Our main result in this section doesn't depend on such a choice though.)

When constructing the morphism string, when we go backwards in homological height we put a dash on the morphism.

\begin{thm}\label{thm-main}
Let $T$ be a positive rational tangle. (The case of negative rational tangles is similar to what we present below.) Then the Khovanov complex $\Kbp{T}$ of $T$ has a representative that is a zig-zag complex, and the morphism string associated to this minimal complex is a word $w$ in $\{ a,b,c,d,s \}$, possibly with dashes on the letters satisfying the following condition. After removing the dashes from $w$, if $\tilde w = l_1l_2l_3$ is any subword of $w$ consisting of three adjacent letters:
\begin{itemize}
	\item if $l_2 = a$, then $\tilde w = bab$,
	\item if $l_2 = b$, then $l_1 \in \{a,s'\}, \; l_3 \in \{ a,s\}$,
	\item if $l_2 = c$, then $\tilde w = dcd$,
	\item if $l_2 = d$, then $l_1 \in \{ c,s\}, \; l_3 \in \{c,s'\}$,
	\item if $l_2 = s'$, then $\tilde w = ds'b$,
	\item if $l_2 = s$, then $\tilde w = bsd$.
\end{itemize}

\noindent Furthermore, the morphism string of the minimal complex of $\Kbp{T + [1]}$ or $\Kbp{T * [1]}$ can be obtained from the morphism string of the minimal complex of $\Kbp{T}$ by the following rules.
\\ \\
To obtain $\Kbp{T + [1]}$ from $\Kbp{T}$, split the morphism string of the minimal complex of $\Kbp{T}$ into a list of subwords $w_1,\ldots,w_n$ (so that their concatenation $w_1\cdots w_n$ is the morphism string) such that $w_i \in \{c,$ $c'$, $d$, $d'$, $s'\Box$, $s'\Box s$, $ \Box s \}$ where $\Box$ is a string in $\{ a,a',b,b'\}$. The morphism string of $\Kbp{T + [1]}$ is given by the concatenation $f(w_1)\cdots f(w_n)$ where $f$ is the following collection of rules.
\begin{itemize}
	\item If $\boxdot$ is the string obtained from $\Box$ by replacing each letter $a/a'/b/b'$ with $b/b'/a/a'$ respectively, then
	\begin{itemize}
		\item $f(s'\Box) = s'b'\boxdot$,
		\item $f(s'\Box s) = s'b'\boxdot bs$,
		\item $f(\Box s) = \boxdot bs$.
	\end{itemize}
	\item If $w_i$ = $d$, $d'$,
	\begin{itemize}
		\item $f(w_i) = sw_i$ \quad $(i = 1)$,
		\item $f(w_i) = w_i$ \quad $(1<i<n)$,
		\item $f(w_i) = w_is'$ \quad $(i = n)$.
	\end{itemize}
	\item If $w_i = c,c'$, $f(c) = s'bs$, $f(c') = s'b's$.
\end{itemize}

\noindent To obtain $\Kbp{T * [1]}$ from $\Kbp{T}$, split the morphism string of the minimal complex of $\Kbp{T}$ into a list of subwords $w_1,\ldots,w_n$ such that $w_i \in \{a,$ $a'$, $b$, $b'$, $s\Box$, $s\Box s'$, $ \Box s' \}$ where $\Box$ is a string in $\{ c,c',d,d'\}$. The morphism string of $\Kbp{T * [1]}$ is given by the concatenation $g(w_1)\cdots g(w_n)$ where $g$ is the following collection of rules.
\begin{itemize}
	\item If $\boxdot$ is the string obtained from $\Box$ by replacing each letter $c/c'/d/d'$ with $d/d'/c/c'$ respectively, then
	\begin{itemize}
		\item $g(s\Box) = sd\boxdot$,
		\item $g(s\Box s') = sd\boxdot d's'$,
		\item $g(\Box s') = \boxdot d's'$.
	\end{itemize}
	\item If $w_i$ = $b$, $b'$,
	\begin{itemize}
		\item $g(w_i) = s'w_i$ \quad $(i = 1)$,
		\item $g(w_i) = w_i$ \quad $(1<i<n)$,
		\item $g(w_i) = w_is$ \quad $(i = n)$.
	\end{itemize}
	\item If $w_i = a,a'$, $g(a) = sds'$, $g(a') = sd's'$.
\end{itemize}
\end{thm}
Before embarking on the proof, we take a moment to reassure ourselves that the statement is believable.

As a sanity-check we wrote a simple program in Python that calculates $\Kb{T}$ for positive rational tangles $T$ based on our rules. We then extended the program to calculate $Kh(R)$ for rational $R$. We checked the results for the first 25 rational knots (admittedly, by hand --- otherwise it would not have been simple); they were all correct.

\begin{example}\label{eg-82}
We can also run Theorem~\ref{thm-main} by hand, computing $Kh(8_2)$, which would be very intimidating to do by hand with previous technology.

The knot $8_2$ is rational, obtained via numerator closure (Section~\ref{sec-RTs}) from $\la 5,1,2 \ra$. Computing $\Kb{8_2}$ is then just a matter of computing $\Kb{5,1,2}$ before placing the minimal complex in the numerator closure planar arc diagram. Let us compute the structure of the chain complex first; gradings and homological heights will be calculated at the end using a later result.

Morphism strings of Khovanov complexes of rational tangles are not particularly readable, in that the structure of the complex is not as apparent as when it is represented as a dot diagram. As such we'll replace any instance of $s$ with a negative-sloping diagonal line, and an $s'$ with a positive sloping diagonal line.

The morphism string corresponding to $\Kb{1}$ is just $s$. By applying the rules above, we build the morphism string of $\Kb{5,1,2}$ from $s$ as follows.

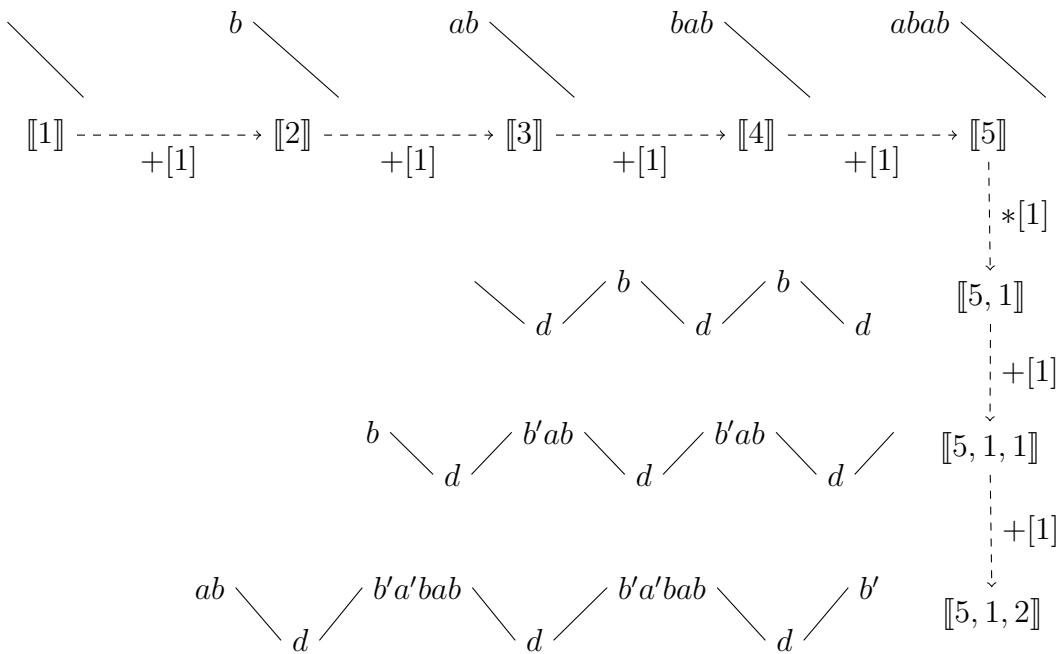
\[
	\begin{tikzpicture}[auto,scale=1]
\begin{scope}[yshift=3cm]
	\draw (0,1) -- (1,0);
	\node (1) at (0.5,-0.5) {$\Kb{1}$};
	\begin{scope}[xshift=3cm]
	\node (11) at (0,1) {$b$};
	\node (12) at (1.5,0) {};
	\draw[-] (11.east) to (12.west);
	\node (13) at (0.75,-0.5) {$\Kb{2}$};
	\end{scope}
	\begin{scope}[xshift=6cm]
	\node (21) at (0,1) {$ab$};
	\node (22) at (1.6,0) {};
	\draw[-] (21.east) to (22.west);
	\node (23) at (0.8,-0.5) {$\Kb{3}$};
	\end{scope}
	\begin{scope}[xshift=9cm]
	\node (31) at (0,1) {$bab$};
	\node (32) at (1.7,0) {};
	\draw[-] (31.east) to (32.west);
	\node (33) at (0.85,-0.5) {$\Kb{4}$};
	\end{scope}
	\begin{scope}[xshift=12cm]
	\node (41) at (0,1) {$abab$};
	\node (42) at (1.8,0) {};
	\draw[-] (41.east) to (42.west);
	\node (43) at (0.9,-0.5) {$\Kb{5}$};
	\end{scope}
	\begin{scope}[dashed]
	\draw[->] (1) to node[anchor=north] {$+[1]$} (13);
	\draw[->] (13) to node[anchor=north] {$+[1]$} (23);
	\draw[->] (23) to node[anchor=north] {$+[1]$} (33);
	\draw[->] (33) to node[anchor=north] {$+[1]$} (43);
	\end{scope}
\end{scope}

\begin{scope}[xshift=6cm,scale=0.7]
	\node (1) at (0,.8) {};
	\node (2) at (1.5,0) {$d$};
	\node (3) at (3,.8) {$b$};
	\node (4) at (4.5,0) {$d$};
	\node (5) at (6,.8) {$b$};
	\node (6) at (7.5,0) {$d$};
	\draw[-] (1.east) to (2.west);
	\draw[-] (2.east) to (3.west);
	\draw[-] (3.east) to (4.west);
	\draw[-] (4.east) to (5.west);
	\draw[-] (5.east) to (6.west);
	\node (10) at (9.9,0.5) {$\Kb{5,1}$};
	\end{scope}
\draw[->,dashed] (43) to node[anchor=west] {$*[1]$} (10);
	\begin{scope}[shift={(4.8,-2)},scale=0.7]
	\node (11) at (0,.8) {$b$};
	\node (12) at (1.5,0) {$d$};
	\node (13) at (3.3,.8) {$b'ab$};
	\node (14) at (5.1,0) {$d$};
	\node (15) at (6.9,.8) {$b'ab$};
	\node (16) at (8.7,0) {$d$};
	\node (17) at (10,.8) {};
	\draw[-] (11.east) to (12.west);
	\draw[-] (12.east) to (13.west);
	\draw[-] (13.east) to (14.west);
	\draw[-] (14.east) to (15.west);
	\draw[-] (15.east) to (16.west);
	\draw[-] (16.east) to (17.west);
	\node (23) at (11.6,0.5) {$\Kb{5,1,1}$};	
	\end{scope}
\draw[->,dashed] (10) to node[anchor=west] {$+[1]$} (23);
\begin{scope}[scale=0.7,shift={(4,-6)}]
	\node (1) at (-0.2,1) {$ab$};
	\node (2) at (1.5,0) {$d$};
	\node (3) at (3.7,1) {$b'a'bab$};
	\node (4) at (5.9,0) {$d$};
	\node (5) at (8.3,1) {$b'a'bab$};
	\node (6) at (10.6,0) {$d$};
	\node (7) at (12.2,1) {$b'$};
	\draw[-] (1.east) to (2.west);
	\draw[-] (2.east) to (3.west);
	\draw[-] (3.east) to (4.west);
	\draw[-] (4.east) to (5.west);
	\draw[-] (5.east) to (6.west);
	\draw[-] (6.east) to (7.west);
	\node (20) at (14.5,0.5) {$\Kb{5,1,2}$};
\end{scope}
\draw[->,dashed] (23) to node[anchor=west] {$+[1]$} (20);
\end{tikzpicture}
\]
All that's left to do is to take the numerator closure of this complex and determine the homological and grading information. The former is easier to do if we construct the dot diagram of the complex. This and its closure are in Figure~\ref{fig-82-dot-diag} below.

\begin{figure}[ht!]
	\centering
	\begin{tikzpicture}[very thick,scale=.7]
\begin{scope}
\foreach \x in {0,...,7}
{
	\draw[gray,xshift=0.5cm] (\x,-0.5) -- (\x,6);
}
\draw[yscale=0.25] (0,0) -- ++(4,4) -- ++ (-3,3) -- ++ (5,5) -- ++(-3,3) -- 	++(5,5) -- ++(-2,2);
\foreach \p in {(0,0),(1,0.25),(2,0.5),(3,1.25),(2,1.5),(1,1.75),(2,2),(3,2.25),(4,2.5),(5,3.25),(4,3.5),(3,3.75),(4,4),(5,4.25),(6,4.5),(7,5.25),(6,5.5)}
{\filldraw[fill=white,very thick] \p circle (4.5pt);}
\foreach \p in {(3,0.75),(4,1),(5,2.75),(6,3),(7,4.75),(8,5)}
{\filldraw \p circle (2pt);}
\end{scope}

\begin{scope}[xshift=11cm]
\foreach \x in {0,...,7}
{
	\draw[gray,xshift=0.5cm] (\x,-0.5) -- (\x,6);
}
\draw[yscale=0.25] (0,0) -- (1,1) (2,2) -- (4,4) -- (3,5) (2,6) -- (1,7) (2,8) -- (3,9) (4,10) -- (6,12) -- (5,13) (4,14) -- (3,15) (4,16) -- (5,17) (6,18) -- (8,20) -- (7,21);
\foreach \p in {(0,0),(1,0.25),(2,0.5),(3,1.25),(2,1.5),(1,1.75),(2,2),(3,2.25),(4,2.5),(5,3.25),(4,3.5),(3,3.75),(4,4),(5,4.25),(6,4.5),(7,5.25),(6,5.5)}
{\filldraw[fill=white,very thick] \p circle (4.5pt);}
\foreach \p in {(3,0.75),(4,1),(5,2.75),(6,3),(7,4.75),(8,5)}
{\filldraw \p circle (2pt);}
\end{scope}
\end{tikzpicture}
	\caption{The dot diagram of $\Kb{5,1,2}$ and $N(\Kb{5,1,2})$. This example is illustrative of the Khovanov homology of rational knots in general.}
	\label{fig-82-dot-diag}
\end{figure}
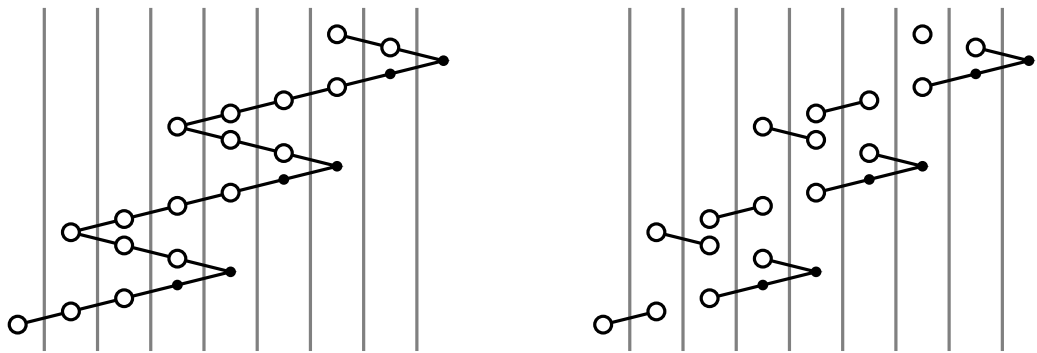

The dot diagram illustrating the closure of the complex has three types of components. The \zdashz\ components have the following form. (The maps at the beginning and end are zero.)

\[
	\begin{tikzpicture}[auto,semithick,scale=1.2]
\node (-1) at (-1,0) {$\bullet$};
\node (0) at (1,0) {\tikz \draw[very thick] (0,0) circle (0.35cm);};
\node (0b) at (1,-0.65) {\footnotesize $(a)$};
\node (1) at (3,0) {\tikz \draw[very thick] (0,0) circle (0.35cm);};
\node (1b) at (3,-0.65) {\footnotesize $(a+2)$};
\node (2) at (5,0) {$\bullet$};
\draw[->] (-1) to (0);
\draw[->] (0) to node {$2\;\nodedottedcircle$} (1);
\draw[->] (1) to (2);
\end{tikzpicture}
\]
The homology groups of these complexes are trivial to compute. After delooping and applying the standard TQFT functor, taking homology gives
\[
	\Kh^n(\zdashz) = \Z_{(a-1)}, \qquad \Kh^{n+1}(\zdashz) = \Z_{(a+3)} \oplus \Z/2\Z_{(a+1)}.
\]

The wedge-shaped complexes in the dot diagram with four subobjects are homotopy equivalent to \zdashz\ components in one higher homological height.\footnote{This follows from a square isomorphism, that the Khovanov bracket is invariant under R1, and that planar algebra morphisms send homotopic equivalent complexes to homotopic equivalent complexes.} No prizes for guessing what complex the lone circle in the diagram is.

Therefore $\Kb{8_2}$ splits into a direct sum of nine complexes, and we have computed the homology of these complexes. To finish, we simply need to calculate the homological heights and gradings. Using a later result, by the matrix multiplication in Example~\ref{bigrading-example}, we find that the subobject with the lowest bigrading is a $[\infty]$ tangle with quantum grading $-16$ and homological height $-6$.

For all intents and purposes we're done, although it's common to write the result as the Khovanov polynomial instead of the actual Khovanov homology. The Khovanov polynomial of a link is a two-variable Laurent polynomial in $\Z[q^{\pm 1},t^{\pm 1}]$ defined by
\begin{equation}
	\sum_j \qdim (Kh^j(L)) \cdot t^j.
\end{equation}

It's common to denote the Khovanov polynomial in the form of a table. The positive integer in $i$-th $j$-th square of the table corresponds to the coefficient of $t^j q^i$ in the polynomial.

Filling the table in using the information obtained above is a trivial matter; \zdashz\ components in the dot diagram correspond to knights moves while $\circleprop$ corresponds to the exceptional pair. If we compare the table to that for $8_2$ on the Knot Atlas, \cite{KA} we see that the two agree.

\begin{table}[ht]
	\centering
	\begin{tabular}{|c|c|c|c|c|c|c|c|c|c|}
	\hline
	$i \setminus j$ & $-6$&$-5$&$-4$&$-3$&$-2$&$-1$&$0$&$1$&$2$\\
	\hline
	$1$  & & & & & & & & &1\\
	\hline
	$-1$ & & & & & & & & & \\
	\hline
	$-3$ & & & & & & &2&1& \\
	\hline
	$-5$ & & & & & &1&1& & \\
	\hline
	$-7$ & & & & &2&1& & & \\
	\hline
	$-9$ & & & &1&1& & & & \\
	\hline
	$-11$& & &1&2& & & & & \\
	\hline
	$-13$& &1&1& & & & & & \\
	\hline
	$-15$& &1& & & & & & & \\
	\hline
	$-17$&1& & & & & & & & \\
	\hline
	\end{tabular}
	\caption{The Khovanov homology of $Kh(8_2)$, by hand!}
	\end{table}
\end{example}
We now return to the main result of this section.
\\ \\
\emph{Proof of Theorem~\ref{thm-main}:}
The first claim of the theorem regarding the structure of $\Kb{T}$ for rational $T$ follows from induction from the subword rules of the theorem. Namely, since every rational tangle is given by a finite sequence of additions and products of $[1]$, one simply shows that the rules preserve the property of the first claim.

The proof of the subword rules essentially follows from square isomorphisms and the proof of the minimal complex for integer tangles. The proof describing how the morphism string of $\Kb{T + [1]}$ is obtained from $\Kb{T}$ is analogous to that for $\Kb{T * [1]}$ case; as such we prove the former.

To save space, we don't replicate the complexes used in the proof of Theorem~\ref{thm-mat-actions}, but since that proof and this are essentially two sides of the same coin, we refer the reader to them.

Let us first explain the general structure of $\Kb{T + [1]}$. If $\Kb{T}$ is a zig-zag complex, then we can write the complex as a horizontal string of $[0]$ or $[\infty]$ tangles with forwards or backwards arrows between the subobjects. If the morphism string of $\Kb{T}$ contains $n$ letters, can view $\Kb{T + [1]}$ as a `$2 \times (n + 1)$ complex'. That is, by constructing the complex using the addition planar arc diagram, we have $\Kb{T + [1]} = \Kb{D([T],[1])} = D(\Kb{T},\Kb{1})$. This is a complex consisting of two rows, each with $n + 1$ subobjects, and various left, right, and down arrows between adjacent subobjects in the complex. We now simplify the complex.

Assume the morphism string of $\Kb{T}$ satisfies the rules in the first part of the theorem. Then note that any maps $D(d,1)$ in the top row simplify to zero, and maps $D(d,1)$ in the bottom row are just $d$. We can thus view the complex as a series of chains of anticommuting squares with non-zero maps which are `connected' by $d$ maps. (Example~\ref{eg-5-1} illustrates the analogous case for $\Kb{T * [1]}$, where the chains of squares are connected by $b$ maps.) If the morphism string ends or begins with a $d$ map, then there an extra saddle map is created, which is the second subword rule regarding $\Kb{T + [1]}$.

By virtue of the structure of the morphism string of $\Kb{T}$, these chains of squares come in two types: a single square induced by a morphism string $c/c'$, or a chain of squares induced by a morphism string which is a substring of $s'b'a'b'\ldots ababs$ containing at least one $s$ or $s'$.

The single squares simplify via the square isomorphism to squares with morphism string $s'bs$ or $s'b's$, depending on the direction of the arrow in the top of the square. This is the third subword rule regarding $\Kb{T + [1]}$.

For the other type of square chain, note that the morphism string that induced this complex is that of two minimal integer tangle complexes which share the same first subobject. (E.g. $s'b'a'babs$ is the direct sum of $\Kb{3}$ and $\Kb{4}$ with the first subobjects identified.) As such, these square chains simplify down in exactly the same way as in the proof of Proposition~\ref{proof-ints}. That is, morphism strings of the form $ba\ldots bas$ go to $aba\ldots bas$. This is second rule regarding $\Kb{T + [1]}$.

And we're essentially done: each of these chains of square have simplified to zig-zag complexes whose z-ends are $[\infty]$ tangles, and were joined by $d$ maps, so the entire complex is also a zig-zag complex. Since every rational tangle can be constructed from $[1]$ which is a zig-zag complex satisfying the rules in the first part of the theorem, it follows that any complex created by adding or multiplying $[1]$ is a zig-zag complex satisfying the rules too.
\hfill $\Box$

\section{A matrix action on the Bigradings of the minimal complexes}\label{sec-burau}
%!TEX root = ..\main.tex

So far, we have only presented one half the picture: we have essentially ignored the bigradings of the subobjects in the minimal Khovanov complex of a rational tangle. Quite surprisingly, it turns out that the bigradings can be described by matrix actions. Let us formalize what we mean by `bigrading information'. Let $\BN(4)$ be the graded subcategory of $\Mat(\Cobbl)$ generated by the $[\infty]$ and $[0]$ tangles.

Define a function $\Psi:\Kom(BN(4)) \rightarrow \Z[q^{\pm 1},t^{\pm 1}] \la \minizero, \miniinf \ra$ as follows. Given a complex $\Omega$, express each of its objects as direct sum of indecomposable objects. To each indecomposable object $X$ in $\Omega^j$ we associate the element $q^it^j X$ where $i$ is the internal grading of the subobject. Then define $\Psi(C)$ to be the sum of these elements over all subobjects in the complex.

We will present an element $p_0 \;\minizero + p_{\infty} \miniinf \in \Z[q^{\pm 1},t^{\pm 1}] \la \minizero, \miniinf \ra$ as $(p_0,p_{\infty})$. For example,
\[
	\Psi(\bullet \rightarrow \underline{\miniinf}_{(0)} \stackrel{\phi}{\longrightarrow} \miniinf_{(0)} \rightarrow \bullet) = (0,1 + t).
\]

As the example demonstrates, $\Psi$ is not a homotopy invariant, since this complex is contractible when $\phi$ is an isomorphism. As such, to make the map well-defined we will only consider complexes in $\Kom(\BN(4))$ up to isomorphism.

Let us now consider rational tangles as partial closures of elements in $B_3$. (Figure~\ref{fig-3braid_pres}.) More precisely, we can consider a positive rational tangle $T = \la a_1,a_2,\ldots,a_n \ra$ (with $a_i \geq 0$ ($a_i \leq 0)$) as the element $\sigma_1^{a_1}\sigma_2^{-a_2} \sigma_1^{a_3} \sigma_2^{-a_4}\cdots $ of $B_3$ whose bottom two rows have been closed off at the beginning.

From such a perspective, adding $[1]$ to a rational tangle corresponds to multiplying the braid by $\sigma_1$, while multiplying the tangle by $[1]$ corresponds to multiplying the braid by $\sigma_2\inv$.
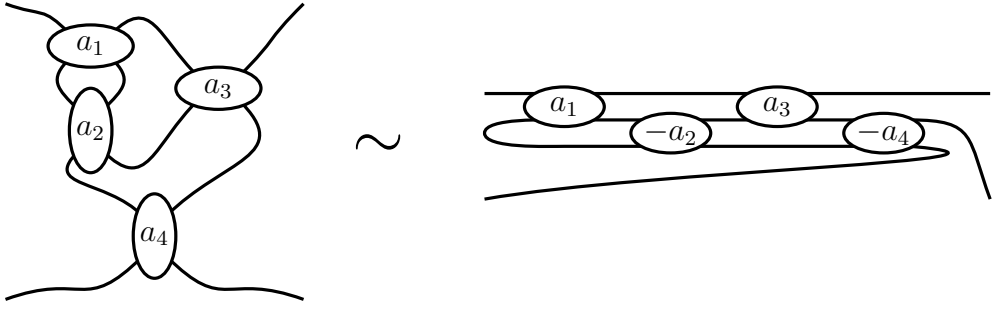
\begin{figure}
	\centering
	\begin{tikzpicture}[scale=0.7,very thick]
\begin{scope}[looseness=1.5,scale=0.8]
\node (a1l) at (1.8,8) {};
\node (a1r) at (2.2,8) {};
\node (a2u) at (2,6.2) {};
\node (a2d) at (2,5.8) {};
\node (a3l) at (4.8,7) {};
\node (a3r) at (5.2,7) {};
\node (a4u) at (3.5,3.7) {};
\node (a4d) at (3.5,3.3) {};

\draw (a1l) [out=135,in=-20] to (0,9);
\draw (a1l) [out=-135,in=135] to (a2u);
\draw (a1r) [out=-45,in=45] to (a2u);
\draw (a2d) [out=-135, in=135] to (a4u);
\draw (a2d) [out=-45,in=-135] to (a3l);
\draw (a1r) [out=45,in=135] to (a3l);
\draw (a3r) [out=-45,in=45] to (a4u);
\draw (a4d) [out=-135,in=20] to (0,2);
\draw (a4d) [out=-45,in=160] to (7,2);
\draw (a3r) [out=45,in=-135] to (7,9);

\filldraw[fill=white] (2,8) ellipse (1cm and .5cm) node {$a_1$};
\filldraw[fill=white] (2,6) ellipse (.5cm and 1cm) node {$a_2$};
\filldraw[fill=white] (5,7) ellipse (1cm and .5cm) node {$a_3$};
\filldraw[fill=white] (3.5,3.5) ellipse (.5cm and 1cm) node {$a_4$};
\end{scope}

\begin{scope}[shift={(9,4.5)},scale=0.5]
\draw[looseness=1] (0,-2) [out=10,in=-5] to (16,0) -- (3,0) [looseness=0.5,out=180,in=-90] to (0,0.5) [out=90,in=180] to (3,1) -- (15,1);
\draw (0,2) to (19,2);
\draw (16,1) [looseness=1.5,out=0,in=110] to (19,-2);

\filldraw[fill=white] (3,1.5) ellipse (1.5cm and .75cm) node {$a_1$};
\filldraw[fill=white] (7,0.5) ellipse (1.5cm and .75cm) node {$-a_2$};
\filldraw[fill=white] (11,1.5) ellipse (1.5cm and .75cm) node {$a_3$};
\filldraw[fill=white] (15,0.5) ellipse (1.5cm and .75cm) node {$-a_4$};
\end{scope}
\node at (7,4.5) {\huge $\sim$};
\end{tikzpicture}
	\caption{Rational tangles can be considered as (partial closures) of elements of $B_3$.}
	\label{fig-3braid_pres}
\end{figure}

This clearly isn't the traditional way to close off a braid. Rather, our `closure' of a braid is really the numerator closure of the corresponding rational tangle. Since every rational tangle has a unique canonical form, it follows that each positive (negative) rational tangle has a unique presentation $\sigma_1^{a_1} \sigma_2^{-a_2}\sigma_1^{a_3}\cdots \sigma_1^{a_n}$ with $a_i > 0$ ($a_i < 0$) for $i<n$ and $a_n \geq 0$ ($a_n \leq 0$).

We can then describe the bigrading information of the Khovanov complexes of positive rational tangles as follows. (Negative rational tangles admit a similar description.)

\begin{thm}
\label{thm-mat-actions}
Let $\phi:\Q \rightarrow B_3$ be the function taking a rational number $r$ to the canonical braid presentation (above) corresponding to the rational tangle $T$ with $F(T) = r$. Let $\Psi$ be as above. For any rational tangle $T$, identify the minimal complex $\Kbp{T}$ with its representative in $\Kom(\BN(4))$. Finally, let $r \in \Q_+$, and fix an orientation of $\phi(r)$. Then $\Psi(\Kbp{\phi(r)\sigma_1})$ is given by $\Psi(\Kbp{\phi(r)})A_{\pm}$ and $\Psi(\Kbp{\phi(r)\sigma_2\inv})$ is given by $\Psi(\Kbp{\phi(r)})P_{\pm}$, where
\[
	A_+ = q\begin{pmatrix} qt & 1 \\ 0 & q\inv	\end{pmatrix}, \qquad A_- = q^{-2}t^{-1}\begin{pmatrix} qt & 1 \\ 0 & q\inv	\end{pmatrix},
\]

\[
	P_+ = q^2\begin{pmatrix} qt & 0 \\ t & q\inv	\end{pmatrix}, \qquad P_- = q\inv t\inv 
\begin{pmatrix} qt & 0 \\ t & q\inv	\end{pmatrix}.
\]
(The signs correspond to the orientation of the crossing being added.)
\end{thm}

We note that this is quite unexpected. While $B_3$ acts on rational tangles as described, there is no reason to expect that the bigrading information of the minimal Khovanov complexes should admit such a description.

\begin{example}
	\label{bigrading-example}
	Let us do the bigrading calculation we used earlier when computing $\Kh(8_2)$ in Example~\ref{eg-82}. The knot $8_2$ is the numerator closure of $\la 5,1, 2\ra$. When $8_2$ is oriented the underlying tangle inherits an orientation of type I (recall Section~\ref{sec-RTs}). As such, when we build $\la 5,1,2 \ra$ from $[0]$, the first six crossings have negative orientation, while the last two crossings are positively oriented.

	As such, the bigrading information of $\Kb{5,1,2}$ is given by
	\[
		(1,0) \cdot {A_-}^5{P_-}{A_+}^2 = (p_0,p_{\infty}).
	\]
	Calculating these, we find
	\[
		p_0 = q^{-11}t^{-3} + q^{-9}t^{-2} + q^{-7}t\inv + q^{-5} + q^{-3}t + q\inv t^2,
	\]
	\[
		p_{\infty} = q^{-16}t^{-6} + 2q^{-12}t^{-5} + 3q^{-12}t^{-4} + 3q^{-10}t^{-3} + 3q^{-8}t^{-2} + 2q^{-6}t\inv + 2q^{-4} + q^{-2}t.
	\]
	By examining Figure~\ref{fig-82-dot-diag} we see that bigrading information of the subobjects in the complex are accounted for.
\end{example}

\emph{Proof of Theorem~\ref{thm-mat-actions}:}
We will prove that $\Psi(\Kb{\phi(r)\sigma_1})$ is given by $\Psi(\Kb{\phi(r)})A_{\pm}$. (The proof of the other claim is analogous.) Let $T$ be a positive rational tangle, and consider $\Kb{T}$. Adding $[1]$ to $T$ is the same as multiplying the corresponding braid by $\sigma_1$. We will prove the claim by examining how the bigradings of the subobjects in $\Kb{T}$ influence the bigradings of the subobjects in $\Kb{T + [1]}$. By doing this we will be able to express the bigrading information of $\Kb{T+ [1]}$ in terms of $\Kb{T}$.

From Theorem~\ref{thm-main} we can uniquely split the morphism string of $\Kb{T}$ into subwords in $\left\{ c, c',d,d',s'\Box,s'\Box s,\Box s \right\}$. We can account for all the subobjects in $\Kb{T + [1]}$ by accounting for the subobjects created from each subword of $\Kb{T}$, though when we do this we need to be careful not to over-count or miss any subobjects.

As in the last section, construct $\Kb{T+[1]}$ from $\Kb{T}$ using the `addition' planar arc diagram $D$. Let us further assume that the orientation of the crossing being added is positive. (So we will need to show the bigradings change in a manner described by the action of $A_+$.)

Theorem~\ref{thm-main} tells us that a word in the morphism string of $\Kb{T}$ of the form $s'\Box s$ is transformed to $s'b'\boxdot b s$. (In this instance $\Box$ is an alternating word in $\{ a,a',b,b'\}$.) Let us recall the proof. The complex associated to $s'\Box s$ has the form
\[
	\begin{tikzpicture}[auto,semithick, scale=1.2]
\node (1) at (-2.5,0) {\nodezero};
\node (2) at (0,0) {\nodeinf};
\node (3) at (2,0) {$\cdots$};
\node (4) at (4,0) {\nodeinf};
\node (5) at (6.5,0) {\nodeinf};
\node (6) at (9,0) {\nodezero\;.};
\begin{scope}[yshift=-0.65cm]
\node (1a) at (-2.5,0) {\footnotesize $(a)$};
\node (2a) at (0,0) {\footnotesize $(a-1)$};
\node (4a) at (4,0) {\footnotesize $(b-3)$};
\node (5a) at (6.5,0) {\footnotesize $(b-1)$};
\node (6a) at (8.95,0) {\footnotesize $(b)$};
\end{scope}
\draw[<-] (1) to node {$s$} (2);
\draw[<-] (2) to (3);
\draw[->] (3) to (4);
\draw[->] (4) to node {$b$} (5);
\draw[->] (5) to node {$s$} (6);
\end{tikzpicture}
\]
This consists of $n$ $\miniinf$ subobjects and $2$ $\minizero$ subobjects.

Recall that when $[1]$ is positively oriented, $\Kb{1} = \bullet \rightarrow \underline{\miniinf}_{(1)} \longrightarrow \minizero_{(2)} \rightarrow \bullet$. Hence when $D(\Kb{T},\Kb{1})$ is constructed, the previous complex segment gives rise to the following complex segment.
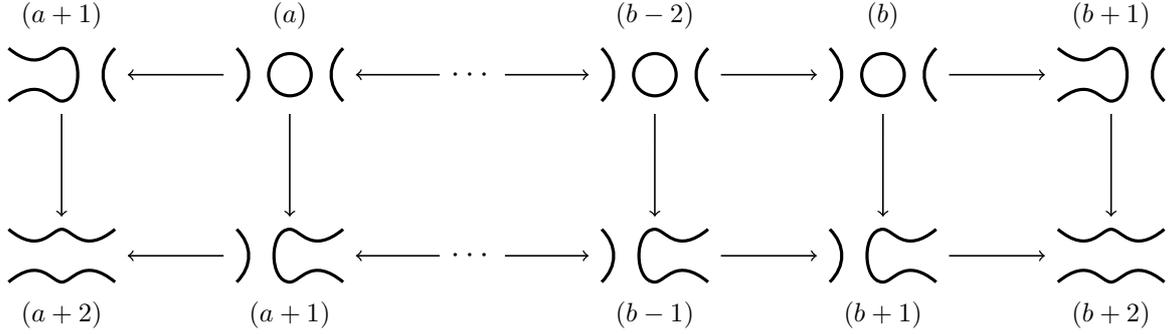
\[
	\begin{tikzpicture}[scale=1.2,semithick]
\node (1) at (0,2) {\nodemhzeroinf};
\node (2) at (2.5,2) {\nodeclownh};
\node (3) at (4.5,2) {$\cdots$};
\node (4) at (6.5,2) {\nodeclownh};
\node (5) at (9,2) {\nodeclownh};
\node (6) at (11.5,2) {\nodemhzeroinf};

\node (7) at (0,0) {\nodemhzerozero};
\node (8) at (2.5,0) {\nodemhinfzero};
\node (9) at (4.5,0) {$\cdots$};
\node (10) at (6.5,0) {\nodemhinfzero};
\node (11) at (9,0) {\nodemhinfzero};
\node (12) at (11.5,0) {\nodemhzerozero};

\begin{scope}[yshift=2.65cm]
\node (1a) at (0,0) {\footnotesize $(a+1)$};
\node (2a) at (2.5,0) {\footnotesize $(a)$};
\node (4a) at (6.5,0) {\footnotesize $(b-2)$};
\node (5a) at (9,0) {\footnotesize $(b)$};
\node (6a) at (11.5,0) {\footnotesize $(b+1)$};
\end{scope}
\begin{scope}[yshift=-0.65cm]
\node (1b) at (0,0) {\footnotesize $(a+2)$};
\node (2b) at (2.5,0) {\footnotesize $(a+1)$};
\node (4b) at (6.5,0) {\footnotesize $(b-1)$};
\node (5b) at (9,0) {\footnotesize $(b+1)$};
\node (6b) at (11.5,0) {\footnotesize $(b+2)$};
\end{scope}

\draw[<-] (1) to (2);
\draw[<-] (2) to (3);
\draw[->] (3) to (4);
\draw[->] (4) to (5);
\draw[->] (5) to (6);

\draw[->] (1) to (7);
\draw[->] (2) to (8);
\draw[->] (4) to (10);
\draw[->] (5) to (11);
\draw[->] (6) to (12);

\draw[<-] (7) to (8);
\draw[<-] (8) to (9);
\draw[->] (9) to (10);
\draw[->] (10) to (11);
\draw[->] (11) to (12);
\end{tikzpicture}
\]
After multiple applications of delooping and Gaussian elimination, this simplifies down.
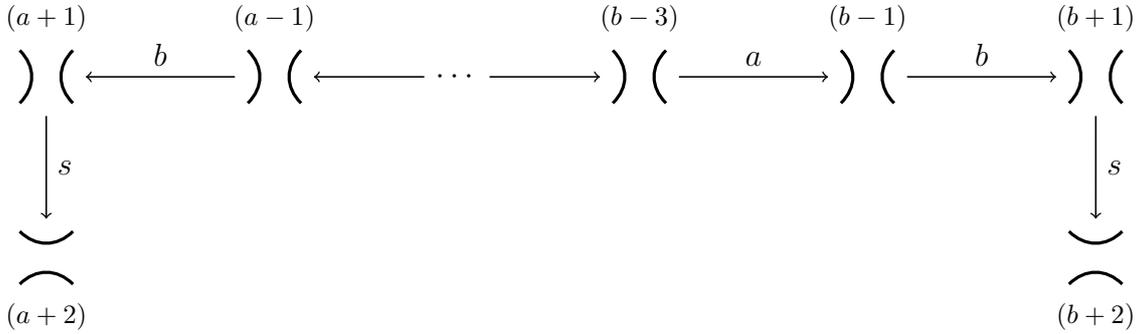
\[
	\begin{tikzpicture}[scale=1.2,auto,semithick]
\node (1) at (0,2) {\nodeinf};
\node (2) at (2.5,2) {\nodeinf};
\node (3) at (4.5,2) {$\cdots$};
\node (4) at (6.5,2) {\nodeinf};
\node (5) at (9,2) {\nodeinf};
\node (6) at (11.5,2) {\nodeinf};

\node (7) at (0,0) {\nodezero};
\node (12) at (11.5,0) {\nodezero};

\begin{scope}[yshift=2.65cm]
\node (1a) at (0,0) {\footnotesize $(a+1)$};
\node (2a) at (2.5,0) {\footnotesize $(a-1)$};
\node (4a) at (6.5,0) {\footnotesize $(b-3)$};
\node (5a) at (9,0) {\footnotesize $(b-1)$};
\node (6a) at (11.5,0) {\footnotesize $(b+1)$};
\end{scope}
\begin{scope}[yshift=-0.65cm]
\node (1b) at (0,0) {\footnotesize $(a+2)$};
\node (6b) at (11.5,0) {\footnotesize $(b+2)$};
\end{scope}

\draw[<-] (1) to node {$b$} (2);
\draw[<-] (2) to (3);
\draw[->] (3) to (4);
\draw[->] (4) to node {$a$} (5);
\draw[->] (5) to node {$b$} (6);

\draw[->] (1) to node {$s$} (7);
\draw[->] (6) to node {$s$} (12);
\end{tikzpicture}
\]
In this portion of the complex, there are now $n+2$ $\miniinf$ subobjects, and still $2$ $\minizero$ subobjects. However, the $\minizero$ subobjects now differ from the $\minizero$ subobjects in the first complex by a factor of $q^2 t$.

On the other hand, of the $n+2$ $\miniinf$ subobjects of the top row above, $n$ were produced after delooping $D(\miniinf,\miniinf)$ terms. These subobjects picked up a quantum grading factor of $q$ from $\Kb{1}$, but then immediately lost it after delooping. The two other $\miniinf$ subobjects were produced via $D(\miniinf,\minizero)$. The bigrading information of these will then differ from those of the $\minizero$ subobjects in the first complex by a factor of $q$.

It follows, then, that the bigrading information $(p'_0,p'_{\infty})$ of this segment can be described in terms of the bigrading information of the first segment $(p_0,p_{\infty})$ via
\[
	(p'_0,p'_{\infty}) = (p_0,p_{\infty}) \cdot \begin{pmatrix} q^2 t & q \\ 0 & 1	\end{pmatrix} = (p_0,p_{\infty}) \cdot A_+.
\]
By the same argument the change in bigrading information described by $s'\Box \rightarrow s'b'\boxdot$ and $\Box s \rightarrow \boxdot b s$ is captured by multiplication by $A_+$.

In between words of the form $s' \Box s$ in the morphism string for $\Kb{T}$ are words of the form $\left\{ d,d,c,c' \right\}$. So consider the subcomplex associated to $c$.
\[
	\begin{tikzpicture}[scale=1.2,semithick,auto]
\node (1) at (0,0) {\nodezero};
\node (2) at (3,0) {\nodezero};
\begin{scope}[yshift=-0.65cm]
\node (1b) at (0,0) {\footnotesize $(a)$};
\node (2b) at (3,0) {\footnotesize $(a+2)$};
\end{scope}
\draw[->] (1) to node {$\vcenter{\hbox{\nodezerodota}} + \vcenter{\hbox{\nodezerodotb}}$} (2);
\end{tikzpicture}
\]
When $D(\Kb{T},\Kb{1})$ is constructed, this gives rise to the following complex segment.
\[
	\begin{tikzpicture}[scale=1.2,semithick,auto]
\node (01) at (0,0) {\nodemhzeroinf};
\node (11) at (3,0) {\nodemhzeroinf};
\node (00) at (0,-2) {\nodemhzerozero};
\node (10) at (3,-2) {\nodemhzerozero};

\begin{scope}[yshift=0.65cm]
\node (1a) at (0,0) {\footnotesize $(a+1)$};
\node (2a) at (3,0) {\footnotesize $(a+3)$};
\end{scope}
\begin{scope}[yshift=-2.65cm]
\node (1b) at (0,0) {\footnotesize $(a+2)$};
\node (2b) at (3,0) {\footnotesize $(a+4)$};
\end{scope}
\draw[->] (01) to node {$2\;\vcenter{\hbox{\nodeinfdota}}$} (11);
\draw[->] (00) to node {$\vcenter{\hbox{\nodezerodota}} + \vcenter{\hbox{\nodezerodotb}}$} (10);
\draw[->] (01) to node {$s$} (00);
\draw[->] (11) to node {$-s$} (10);
\end{tikzpicture}
\]
This simplifies by the square isomorphism.
\[
	\begin{tikzpicture}[scale=1.2,semithick,auto]
\node (01) at (0,0) {\nodeinf};
\node (11) at (2.5,0) {\nodeinf};
\node (00) at (0,-2) {\nodezero};
\node (10) at (2.5,-2) {\nodezero};

\begin{scope}[yshift=0.65cm]
\node (1a) at (0,0) {\footnotesize $(a+1)$};
\node (2a) at (2.5,0) {\footnotesize $(a+3)$};
\end{scope}
\begin{scope}[yshift=-2.65cm]
\node (1b) at (0,0) {\footnotesize $(a+2)$};
\node (2b) at (2.5,0) {\footnotesize $(a+4)$};
\end{scope}
\draw[->] (01) to node {$b$} (11);
\draw[->] (01) to node {$s$} (00);
\draw[->] (11) to node {$-s$} (10);
\end{tikzpicture}
\]
The bigrading of the $\minizero$ subobjects here differ from the originals by a factor of $q^2 t$, while the bigrading of the $\miniinf$ subobjects differ from the originals by a factor of $q$. Hence, on this complex segment the overall change in bigrading information can be described by the action of $A_+$.

We're essentially done: if a subword $d/d'$ in the morphism string is not at either end, then it is adjacent to letters in $\left\{ s',s,c,c' \right\}$, meaning the subobjects and the head and tail of $d$ have already been accounted for by our analysis of the $s' \Box s \rightarrow s' b' \boxdot b s$ and $c \rightarrow s'bs$ transformations. If $d/d'$ is at the end of the morphism string, then
there will be a subobject and the head or tail of it we will not have accounted for yet, but it is easy to check that the bigradings of the two subobjects in $\Kb{T + [1]}$ that this produces can be described by the action of $A_+$ on the bigrading information of the subobject too.

Hence the change in bigrading information of the complex is locally described by the action of $A_+$, and since we have accounted for all the subobjects in $\Kb{T + [1]}$, it follows that it is globally described by the action too. We then immediately get that $\Psi(\Kb{\phi(r)\sigma_1})$ is given by $\Psi(\Kb{\phi(r)})A_-$ when the crossing being added has negative orientation, since $\Psi(\Kb{1})$ when $[1]$ has negative orientation differs from $\Psi(\Kb{1})$ when $[1]$ has positive orientation by a factor of $q^{-3} t\inv$.
\hfill $\Box$

Given that the bigrading information admits such an elegant description in terms of matrix actions and $B_3$, it is a natural question to ask whether the matrices satisfy the braid relation. Since multiplying the braid by $\sigma_1$ corresponds to multiplication by $A_{\pm}$, and multiplying the braid by $\sigma_2\inv$ corresponds to multiplication by $P_{\pm}$, if they matrices did braid we would expect $A_{\pm}$ and $P_{\pm}\inv$ to braid. This is indeed the case. The ramifications of this observation remain unclear.

\begin{coro}\label{coro-braid-rels}
The pairs of matrices ($A_+$, $P_-\inv$) and ($A_-$, $P_+\inv$) satisfy the braid relation. After a change of basis these give the (reduced) Burau representation of $B_3$.

\begin{proof}
With
\[
	p = \begin{pmatrix}
		q & 0 \\
		0 & 1
	\end{pmatrix},
\]
when we change basis we obtain
\[
	\sigma_1' = p\inv A_+ p = 
	\begin{pmatrix}
		q^2 t & 1 \\
		0 & 1	\end{pmatrix}, \qquad \sigma_2' = p\inv P_-\inv p =
	\begin{pmatrix}
		1 & 0\\
		-q^2 t & q^2 t
	\end{pmatrix}.
\]
This is exactly the (reduced) Burau representation of $B_3$ after the change of variables $t \mapsto -q^2 t$.
\end{proof}
\end{coro}

We note however, that while $A_{\pm}$ and $P_{\pm}$ describe the addition or product of $[1]$ with a positive rational tangle, $A_{\pm}\inv$ and $P_{\pm}\inv$ \emph{do not} describe the addition or product of $[-1]$ with a positive rational tangle.

For instance, consider the Khovanov complex associated to $[1]$ with positive orientation. When the corresponding braid ($\sigma_1$) is multiplied by $\sigma_1\inv$, the $[0]$ tangle is obtained. However,
\[
	\Psi(\Kb{1}) \cdot A_-\inv = (q^2 t, q) \cdot \begin{pmatrix} q & -q^2 \\ 0 & q^3 t	\end{pmatrix} = (q^3 t, 0) \ne \Psi(\Kb{0}) = (1,0).
\]
We note though, that the orientations are at fault for the discrepancy here: if $[1]$ is positively oriented, then when the braid is multiplied by $\sigma_1\inv$, the new crossing is negatively oriented. (Indeed, $A_+$ and $P_-$ differ from each other by a factor of $q^3 t$.)

\bibliographystyle{alpha}
\renewcommand{\markboth}[2]{}% Remove header adjustment
\bibliography{bibtotal}

\end{document}